\documentclass[12pt,reqno]{amsart}
\usepackage{longtable,array,booktabs}
\usepackage[T1]{fontenc}
\usepackage{lmodern}
\usepackage{amsmath,amssymb,amsthm,mathrsfs}
\usepackage{aliascnt}
 \usepackage{enumitem}
\usepackage{geometry}
\usepackage{xcolor}

\usepackage{microtype}
\usepackage[colorlinks=true,linkcolor=blue!45!black,citecolor=blue!45!black,urlcolor=blue!45!black]{hyperref}
\usepackage[nameinlink,capitalise,noabbrev]{cleveref}
\allowdisplaybreaks
\setlist[enumerate]{topsep=0.45em,itemsep=0.25em,parsep=0pt}
\setlist[itemize]{topsep=0.45em,itemsep=0.25em,parsep=0pt}
\numberwithin{equation}{section}

\newtheorem{theorem}{Theorem}[section]
\newaliascnt{proposition}{theorem}
\newtheorem{proposition}[proposition]{Proposition}
\aliascntresetthe{proposition}
\newaliascnt{lemma}{theorem}
\newtheorem{lemma}[lemma]{Lemma}
\aliascntresetthe{lemma}
\newaliascnt{corollary}{theorem}
\newtheorem{corollary}[corollary]{Corollary}
\aliascntresetthe{corollary}
\theoremstyle{definition}
\newaliascnt{definition}{theorem}
\newtheorem{definition}[definition]{Definition}
\aliascntresetthe{definition}
\newaliascnt{convention}{theorem}
\newtheorem{convention}[convention]{Convention}
\aliascntresetthe{convention}
\newaliascnt{example}{theorem}
\newtheorem{example}[example]{Example}
\aliascntresetthe{example}
\theoremstyle{remark}
\newaliascnt{remark}{theorem}
\newtheorem{remark}[remark]{Remark}
\aliascntresetthe{remark}

\newtheoremstyle{mainresult}
  {0.75em}
  {0.75em}
  {\normalfont}
  {}
  {\bfseries}
  {.}
  {0.5em}
  {}
\theoremstyle{mainresult}
\newtheorem*{mainthmA}{Theorem A}
\newtheorem*{mainthmB}{Theorem B}
\newtheorem*{mainthmC}{Theorem C}
\newtheorem*{mainthmD}{Theorem D}
\newtheorem*{mainthmE}{Theorem E}
\theoremstyle{remark}

\crefname{proposition}{Proposition}{Propositions}
\Crefname{proposition}{Proposition}{Propositions}
\crefname{lemma}{Lemma}{Lemmas}
\Crefname{lemma}{Lemma}{Lemmas}
\crefname{corollary}{Corollary}{Corollaries}
\Crefname{corollary}{Corollary}{Corollaries}
\crefname{definition}{Definition}{Definitions}
\Crefname{definition}{Definition}{Definitions}
\crefname{convention}{Convention}{Conventions}
\Crefname{convention}{Convention}{Conventions}
\crefname{example}{Example}{Examples}
\Crefname{example}{Example}{Examples}
\crefname{remark}{Remark}{Remarks}
\Crefname{remark}{Remark}{Remarks}
\crefname{scalarlemma}{Lemma}{Lemmas}
\Crefname{scalarlemma}{Lemma}{Lemmas}

\newcommand{\CC}{\mathbb C}
\newcommand{\RR}{\mathbb R}
\newcommand{\ZZ}{\mathbb Z}
\newcommand{\QQ}{\mathbb Q}
\newcommand{\NN}{\mathbb N}
\newcommand{\TT}{\mathbb T}
\newcommand{\CP}{\mathbb{CP}}
\newcommand{\Sphere}{\mathbb S}
\newcommand{\Vol}{\operatorname{Vol}}
\newcommand{\Ind}{\operatorname{Ind}}

\newcommand{\rank}{\operatorname{rank}}
\newcommand{\Span}{\operatorname{span}}
\newcommand{\sgn}{\operatorname{sgn}}
\newcommand{\ord}{\operatorname{ord}}
\newcommand{\lcm}{\operatorname{lcm}}

\newcommand{\Id}{\operatorname{Id}}
\newcommand{\cH}{\mathcal H}
\newcommand{\cL}{\mathcal L}

\newcommand{\cM}{\mathcal M}

\newcommand{\dd}{\,d}
\newcommand{\bmu}{\boldsymbol\mu}

\newcommand{\cE}{\mathcal E}
\newcommand{\cN}{\mathcal N}

\newcommand{\cC}{\mathscr C}
\newcommand{\Diff}{\operatorname{Diff}}

\newcommand{\vol}{\operatorname{vol}}
\newcommand{\tr}{\operatorname{tr}}

\newcommand{\dist}{\operatorname{dist}}

\newcommand{\PaperTitle}%
 {The Geometry and Dynamics of Spiral Minimal Products}
\newcommand{\PaperShortTitle}{Spiral Minimal Products}
\newcommand{\PaperAuthors}{Haizhong Li and Yongsheng Zhang}

\title[\PaperShortTitle]{\small The Geometry and Dynamics of Spiral Minimal Products}
\author{Haizhong Li}
\address{Department of Mathematical Sciences, Tsinghua University, Beijing 100084, P. R. China}
\email{lihz@mail.tsinghua.edu.cn}
\author{Yongsheng Zhang}
\address{Academy for Multidisciplinary Studies, Capital Normal University, Beijing 100048, P. R. China}
\email{yongsheng.chang@gmail.com}
\date{}

\hypersetup{
 pdftitle={\PaperTitle},
 pdfauthor={\PaperAuthors},
 pdfsubject={Minimal submanifolds, phase dynamics, dense closing, Routh--Sturm index growth, and embeddedness},
 pdfkeywords={Spiral minimal products, weighted geodesics, Hamiltonian reduction, Routh reduction, dense closing, Morse index growth, embeddedness, special Legendrian submanifolds, minimal Lagrangians}
}

\begin{document}

       \begin{abstract}
                     We study the spiral product
           $G_\gamma(t,x,y)=(z_1(t)f_1(x),z_2(t)f_2(y))$,
           which couples two spherical $\cC$-totally real immersed factors
           through a profile curve $\gamma=(z_1,z_2)\subset\Sphere^3$.
           Its minimality is governed by a weighted-geodesic system:
           $G_\gamma$ is minimal precisely when both factors are minimal
           and $\gamma$ is an unparametrized geodesic of
           $|z_1|^{2k_1}|z_2|^{2k_2}g_{\Sphere^3}$.
           This flow is Liouville integrable and,
           when $k_1+k_2>0$,
           its phase map has an open dense full-rank locus.
           Consequently,
           closed profiles of arbitrarily large primitive order occur densely.
           For compact minimal factors,
           Routh reduction and Sturm oscillation give
           $\Ind(G_\gamma)\geq 
           \Ind(f_1)+\Ind(f_2)+2m_\gamma-3$
           for a closed oscillatory profile of primitive closing order $m_\gamma$.
           On the contact level,
           a factor-adapted choice of profiles produces,
           from any prescribed pair of compact connected embedded special Legendrians,
           embedded special Legendrian products of every sufficiently large prime closing order.
           Finally,
           canonical finite horizontal lifts and Hopf projection give Delaunay-type minimal Lagrangians in complex projective spaces.
           For compact embedded inputs and ordinarily closed profiles,
           the primitive spherical quotient is embedded,
           while its projective quotient is embedded exactly 
           when the reduced relative winding is one.
       \end{abstract}

\maketitle

\section*{Introduction}
          \label{se1}

            Minimal submanifolds in spheres have long been a central subject
            in differential geometry,
            e.g. \cite{Takahashi1966,Simons1968,Lawson1970}.
            They arise naturally as critical points of area
            and as links of minimal cones in Euclidean space.
     %       Their systematic construction, together with the global questions of closedness, Morse index, and embeddedness, 
        %    remains a basic problem.
%
                Constructing systematic families of closed examples
                and controlling their Morse index and embeddedness
                remain fundamental global problems.

              The spiral product studied here couples two spherical immersions through a profile curve in $\Sphere^3$.
             A priori,
             its minimality equations involve the geometry of all three inputs.
             Theorem A identifies their common structural reduction:
             for minimal $\cC$-totally real factors,
             the full minimality system becomes the geodesic equation of a  weighted metric on $\Sphere^3$,
             and the factors enter this equation only through their dimensions.
                      %The decoupling also has a second-variation counterpart consequence,
             The same decoupling persists in the second variation,
             allowing negative directions from the factors and the profile
             to be combined additively.
              Thus the weighted-geodesic system is the common geometric and dynamical source of the phase dynamics,
             Morse-index estimates,
             factor-adapted embeddedness,
             and projective constructions developed below.
%
   %          The spiral product thereby transfers the dynamics of a curve
      %       into the geometry of a higher-dimensional minimal submanifold:
         %    closed orbits produce closed products,
            % radial oscillations force Morse-index growth,
            % and phase return and factor incidences govern embeddedness.

          Let $Jz=iz$ on a complex Euclidean space and
          $\omega(u,v)=\langle Ju,v\rangle$.  
          On its unit sphere 
          set
          $\alpha_z(v)=\langle Jz,v\rangle$.

       \begin{definition}[$\cC$-totally real immersion]
           \label{df1}
           A smooth immersion $f:M^k\to\Sphere^{2n+1}\subset\CC^{n+1}$
            is \emph{$\cC$-totally real} 
            if $f^*\alpha=0$,
            equivalently,
           \begin{equation}\label{eq1}
                              \langle Jf,df(X)\rangle=0
                               \qquad\text{for every }X\in TM.
           \end{equation}
           Then $f^*\omega=0$ and
           $Jdf(TM)\subset N_fM$. 
           For $k=n$, 
                              this is the Legendrian condition.
       \end{definition}

          Throughout this paper, 
          all manifolds are smooth and without boundary.  
          A connected
          zero-dimensional factor is a point, with volume $1$ and normal
          Morse index $0$.
          Unless stated otherwise,
          the factors
          \[
                              f_i:M_i^{k_i}\longrightarrow \Sphere^{2n_i+1}\subset\CC^{n_i+1},
                               \qquad i=1,2,
          \]
          are spherical $\cC$-totally real immersions.

       \begin{definition}[Spiral product]
           \label{df2}
           Let $I\subset \RR$ be a connected interval
            and $\gamma=(z_1,z_2):I\to\Sphere^3\subset\CC^2$
            an immersed smooth curve.
           Write $a=|z_1|$ and $b=|z_2|$,
            so that $a,b\geq 0$ and $a^2+b^2=1$.
           For the coordinatewise $\TT^2$-action on $\Sphere^3\subset\CC^2$,
            the two \emph{coordinate circles} are
            $\{z_1=0\}$ and $\{z_2=0\}$.
           On every subinterval on which $z_1z_2\ne0$,
            choose real-valued phase lifts so that
           \begin{equation}\label{eq3}
                              z_1(t)=a(t)e^{i\theta_1(t)},\qquad
                              z_2(t)=b(t)e^{i\theta_2(t)}.
           \end{equation}
           The associated \emph{spiral product} is
           \begin{equation}\label{eq4}
            \begin{split}
                                &
                               G_\gamma:I\times M_1\times M_2
                               \longrightarrow\Sphere^{2n_1+2n_2+3},\\
                                &
                               G_\gamma(t,x,y)
                                =\bigl(z_1(t)f_1(x),z_2(t)f_2(y)\bigr).
            \end{split}
           \end{equation}
       \end{definition}

          By  \eqref{eq1}, 
                       the profile and factor directions  are mutually    orthogonal.  
          Hence the induced metric and volume density split blockwise. 
          Let $g_{\Sphere^3}$ be the standard round metric.
          Writing $a=\cos s$, $b=\sin s$, 
               set
          \[
             \rho(s)=(\cos s)^{k_1}(\sin s)^{k_2},\qquad
             \bar g=\rho(s)^2g_{\Sphere^3}.
          \]
          Then
          $\rho(s)|\gamma'|_{g_{\Sphere^3}}=|\gamma'|_{\bar g}$
          is the profile contribution to the volume density of $G_\gamma$.

       \begin{mainthmA}[Weighted-geodesic characterization]
           For the spherical $\cC$-totally real inputs $f_1,f_2$ above,
            an immersed spiral product  is minimal 
            if and only if 
            both factors are minimal and
         $\gamma$ is an unparametrized geodesic of $\bar g$.
       \end{mainthmA}

          We call such a minimal immersion a 
                                                                       \emph{spiral minimal product}
          and
           the corresponding $\bar g$-geodesic a 
                                                                                    \emph{profile}.  
     The profile equation depends only on $k_1, k_2$
     and  hence
                        the same profile solutions uniformly couple
                        any pair of minimal $\cC$-totally real factors
                        of those dimensions.
                        % As a result, the spiral minimal product actually gives an algorithm.
When the first factor is a point and $\theta_2\equiv0$,
the second component undergoes no rotation and is only rescaled by  $b(t)$.
Consequently, $f_2$ can be any minimal immersion into  a unit sphere.
Taking $f_2$ to be the identity embedding of a round sphere
reduces the construction to the classical cohomogeneity-one setting of
Hsiang--Lawson \cite{HsiangLawson1971}
and recovers the spherical rotational family studied by Otsuki
\cite{Otsuki1970}.
          It is worth pointing out that
           every spherical minimal immersion becomes minimal $\cC$-totally real
           via complexification of the ambient Euclidean space,
           so the construction can be iterated.

          In complex Euclidean space,
          Castro--Urbano \cite{CastroUrbano2004}
          constructed special Lagrangian immersions from two minimal
          Legendrian factors and a suitable Lagrangian surface in $\CC^2$.
          In the spherical contact setting,
          Castro--Li--Urbano \cite{CastroLiUrbano2006}
          formulated the block construction \eqref{eq4}
          with two Legendrian factors and a Legendre curve in $\Sphere^3$,
          and derived the profile equation governing contact-minimality,
          including
          the minimal Legendrian case.
          Haskins--Kapouleas \cite{HaskinsKapouleas2012}
          subsequently developed the global closing theory for this construction
          and proved,
          for standard real equatorial sphere inputs,
          that every closed profile in their family yields an embedded image.
           %in particular,
           %with standard real equatorial sphere inputs, they further proved that the resulting images with every closed profiles are all embedded submanifolds.

          Theorem A goes beyond the contact setting just described
          and reveals the decoupling principle underlying the construction.
          Once the two factors are minimal $\cC$-totally real,
          the PDE system expressing the minimality of the spiral product
          \eqref{eq4} reduces to the geodesic equation
          of the weighted $\Sphere^3$.
          The subsequent questions are therefore how abundant closed profiles are,
          how the Morse indices of the resulting spiral minimal products behave,
          and which closed profiles yield embedded spiral minimal products.

            The starting point for this analysis is the $\TT^2$-symmetry of the weighted metric under  
            rotations of the two phase variables. 
            For 
            an
            affinely parametrized
             $\bar g$-geodesic, set
            \[
\cH:=\frac12|\dot\gamma|_{\bar g}^{2},
\qquad
\mu_1=\rho^2\cos^2s\,\dot\theta_1,
\qquad
\mu_2=\rho^2\sin^2s\,\dot\theta_2.
\]
            Here $\mu_1,\mu_2$ are the Noether momenta and $\cH$ is the geodesic energy.
            These three functions are generically independent and pairwise Poisson-commuting first integrals. 
            Hence the profile system is Liouville integrable.
            Replacing $\gamma(t)$ by $\gamma(\lambda t)$ with $\lambda>0$
                 sends
                        $
                              (\mu_1,\mu_2,\cH)
                                   \mapsto
                                   (\lambda\mu_1,\lambda\mu_2,\lambda^2\cH).
                         $
          Thus, on $\mu_1\ne0$, 
                    fix
                                     $\varepsilon_1:=\sgn(\mu_1)$ and introduce the scale-invariant coordinates on this chart
                        \begin{equation}\label{eq5}
                        C_1=\frac{\mu_2}{\mu_1},     \qquad
                        C_2=\frac{2\cH}{\mu_1^2}.
                        \end{equation}
            Fixing $(C_1, C_2)$ reduces $s$-dynamics to a one-dimensional effective-potential equation:
            % and thus performs a pendulum oscillation 
            $s$ oscillates between two turning values, 
            while the phase variables are recovered by quadrature.

       \begin{definition}[Regular oscillatory profiles and complete cells]
           \label{df3}
           A $\bar g$-geodesic profile
            is \emph{regular oscillatory} if its magnitude $s$
            oscillates between $s_-$ and $s_+$, where
            $
                              0<s_-<s_+<\frac{\pi}{2}.
            $
           A \emph{complete cell} is one oriented magnitude
            excursion $s_-\to s_+\to s_-$.
       \end{definition}
     
     \begin{definition}
[Phase map and primitive ordinary closing order]
\label{df4}
Let $\Delta_i(C_1,C_2)$ be the  increment of
$\theta_i$ over one oriented complete cell,
and define
\begin{equation}\label{eq6}
   \Pi(C_1,C_2)
   :=
   \left(
      \frac{\Delta_1(C_1,C_2)}{2\pi},
      \frac{\Delta_2(C_1,C_2)}{2\pi}
   \right).
\end{equation}
The complete-cell return on the phase torus is
$
   \left(e^{i\Delta_1},e^{i\Delta_2}\right).
$
Hence the corresponding profile closes ordinarily after
$m$ cells precisely when
$
   m\Pi(C_1,C_2)\in\ZZ^2.
$
When such an $m$ exists,
the least positive one is denoted by $m_\gamma$
and called the \emph{primitive ordinary closing order}.
\end{definition}

       %   The  normalized chart on $\mu_2\neq 0$ is similar with 
    %      $\left(\frac{1}{C_1},\frac{C_2}{C_1^2}\right)$
%on the overlap of the two charts.
         %An alternative understanding is global.
         %The normalized chart on $\mu_2\neq 0$
        % and  transition map are described in the proof of Theorem \ref{th5}.
       \begin{mainthmB}[Generic full rank and dense closing]
          In the chart $\mu_1\ne0$,
           the regular doubly spiral locus has the two chambers
           $C_1>0$ and $C_1<0$.
          If $k_1+k_2>0$,
           on each chamber, 
           $\Pi$
           is real analytic and its full-rank locus
          $
             \{(C_1,C_2):
                  \rank D\Pi(C_1,C_2)=2\}
          $
           is open and dense.
           Moreover, for every $M\in\NN$,
the parameters determining ordinarily closed profiles
$\gamma$ with $m_\gamma>M$ are dense  in each chamber.
    %      Every nonempty open subset of $\mathscr V$ contains infinitely
      %     many ordinary-closing parameters with unbounded primitive closing
      %     orders.
       \end{mainthmB}

    For compact inputs,
     Theorem B therefore yields densely many closed spiral minimal products
     of unbounded primitive order.
     Every closed minimal submanifold of positive dimension and  codimension in a unit sphere is unstable \cite{Simons1968}.
     The next result quantifies this instability in terms of the primitive
     closing order.
     %Theorem B can generate infinitely many closed minimal submanifolds in spheres (based on compact inputs). 
    %It is well known that closed minimal submanifold of positive dimension and codimension in a round sphere is unstable \cite{Simons1968}.
     %Then how unstable our spiral minimal products could be?
     %Here we derive the following lower bound for the Morse index $\Ind(G_\gamma)$.

       \begin{mainthmC}[Additive index growth]
           Assume that the factors are compact minimal $\cC$-totally real
           immersions 
           and 
           that $\gamma$ is an ordinarily closed regular oscillatory profile 
           with $\mu_1^2+\mu_2^2\ne0$ and primitive closing order $m_\gamma$.  
           Then
           $
             \Ind(G_\gamma)\geq 
             \Ind(f_1)+\Ind(f_2)+2m_\gamma-3.
           $
%           Here $\Ind(f_i)$ is the normal Morse index in the block sphere
%           $\Sphere^{2n_i+1}$.
         %  If the spiral immersion descends after $m_G$ complete cells to a
        %   quotient immersion $\overline G$, then
        %   $\Ind(\overline G)\geq \max\{0,2m_G-3\}$.
       \end{mainthmC}

   The normal Jacobi problem is generally matrix-valued.
   Here the block structure and Routh reduction isolate a scalar periodic Sturm problem, 
   whose negative directions combine orthogonally with those from the factors.

           We now turn to the distinguished contact level
$\mu_1+\mu_2=0$,
equivalently $C_1=-1$ in the chart above.
Here the profile is Legendrian,
          and the product of Legendrian factors is Legendrian.  
          After a constant  phase rotation, 
          a connected minimal Legendrian becomes special  Legendrian \cite{CastroLiUrbano2006}
          and the cone over it is special Lagrangian
          \cite{HarveyLawson1982}.

          The Strominger--Yau--Zaslow picture relates mirror Calabi--Yau manifolds through dual special Lagrangian torus fibrations \cite{StromingerYauZaslow1996}.
          On the singular side,
          special Lagrangian cones and their special Legendrian links
          provide natural local models for conical singularities \cite{JoyceConicalI}.
             At the immersed level, 
             integrable systems and spectral curves yield families
             of special Legendrian tori
             \cite{McIntosh2003,CarberryMcIntosh2004},
             while Haskins--Kapouleas constructed higher-genus links by a gluing method \cite{HaskinsKapouleas2007}.
             By contrast, compact embedded links are substantially harder to obtain. 
             Classical embedded examples 
                                 include the Harvey--Lawson tori \cite{HarveyLawson1982}, 
                                              Joyce's invariant examples \cite{Joyce2002Symmetries}
                                              and 
                                              Haskins's $S^1$-invariant tori \cite{Haskins2004}. 
                  More recently,
                  Heller--Pedit--Ouyang \cite{HellerOuyangPedit2026}
                  proposed a method for constructing compact embedded
                  special Legendrian surfaces in $\Sphere^5$ with unbounded genus.
          
          Embedded special Legendrian links are central to Joyce's theory of conical singularities 
          \cite{JoyceConicalI},
          play a key role in recent bridge constructions 
          \cite{DimlerGaia2026}
          and 
          enter the
          cylindrical tangent-cone uniqueness results 
          in \cite{CollinsLi2023}. 
          In our setting, however, closedness of the profile alone does not ensure embeddedness: 
          even embedded factors may give a self-intersecting product; 
          see \cref{ex2}. 
          Combining finiteness of the factors' incident phase sets with a low-energy estimate, 
          we select prime-order profiles that exclude all unwanted coincidences.
 %          We overcome this %difficulty 
%           through a factor-adapted selection of 
%           profiles.
           %avoid the phase-incidence sets of the prescribed factors
           %can be chosen to avoid the finitely many nontrivial phase incidences. 
           %This yields an infinite family of compact embedded special Legendrian products from every prescribed pair.
         % On the contact slice, prime-order profiles can be chosen to avoid the finitely many nontrivial phase incidences.

       \begin{mainthmD}[Factor-adapted embedded special Legendrian families]
          Let $f_i:L_i^{k_i}\hookrightarrow\Sphere^{2k_i+1}$,
           $i=1,2$, 
           be any pair of compact connected embedded special Legendrians
           with  $k_1+k_2>0$.
          For every sufficiently large prime $\ell$,
           there is a contact profile of primitive ordinary closing order
           $\ell$ whose associated product
          $
             G_\ell: 
             S^1\times L_1\times L_2
              \hookrightarrow\Sphere^{2k_1+2k_2+3}
          $
           is an embedded special Legendrian satisfying
          $
             \Ind(G_\ell)\geq 
              \Ind(f_1)+\Ind(f_2)+2\ell-3.
          $
         Moreover,  
         %the family 
         $\{G_\ell\}$
         has uniformly bounded  second fundamental forms % are uniformly bounded 
         and linear volume growth in $\ell$.
         % The family contains an infinite pairwise noncongruent subfamily.
       \end{mainthmD}

          A refinement of this construction allows selected  nontrivial stabilizer returns
          and produces embedded special Legendrians with nonproduct topology,
          even from factors that are not phase-saturated;
          see \cref{rmmon,expp}.

          Canonical horizontal lifts of compact connected embedded minimal Lagrangians in complex projective spaces are phase-saturated.
            For every 
            %ordinarily 
            closed nonsteady contact profile, 
            their spiral minimal product has a compact embedded primitive spherical quotient,
            which is special Legendrian after a phase rotation;
            see \cref{th17}.
             Representative embedded minimal Lagrangians in complex projective spaces and their lifts are collected in \cref{tb1}.

             The machinery also  applies to 
          compact connected minimal embeddings into round spheres once their ambient Euclidean spaces are complexified, 
          producing embedded minimal $\cC$-totally real families. 
          In this complexified real case,
          the selection of closed profiles can be made universal
          in terms of $(k_1,k_2)$ alone rather than factor-adapted;
          see \cref{th15}.

          %A refinement permits selected nontrivial stabilizer returns.
          %It produces twisted mapping-torus quotients even from
          %non-phase-saturated factors; see \cref{rmmon,expp}.

          The last main result of this paper concerns minimal Lagrangians
          in complex projective spaces.
          Let $\iota_i:L_i^{n_i}\looparrowright\CP^{n_i}$ be connected minimal Lagrangian immersions.
          Then they admit horizontal special Legendrian lifts
          $\widehat L_i\looparrowright\Sphere^{2n_i+1}$;
          see \cref{pr13} for details. 
          Via the spiral product to these lifts along a contact profile $\gamma$
          and 
          then
          taking the Hopf projection,
            we get   an immersion (unique up to relative phase difference of two lifts)
          $
                              \iota_\gamma:=\pi_{\mathrm H}\circ G_\gamma:
                              \RR\times\widehat L_1\times\widehat L_2
                              \longrightarrow
                              \CP^{n_1+n_2+1}.
          $
          The local construction is due to \cite{CastroLiUrbano2006}; the
          point--sphere case recovers Anciaux's examples \cite{Anciaux2006}.
          Globally,
          however,
          the ordinary return order $m_\gamma$,
          spherical return order $m_G$,
          and projective return order $N$ need not coincide.
         % We identify their mapping-torus domains.
            We determine the corresponding quotient domains.
                 For compact embedded inputs and an ordinarily closed profile,
                 the phase-saturated lifts give
             $
N \mid m_G \mid m_\gamma.
$
          %  Use $\overline \iota_\gamma$ for the primitive projective quotient of $\iota_\gamma$.
                 Thus ${\iota}_\gamma$ descends after $N$ cells to
                 ${\overline{\iota}_\gamma}:
                 \mathcal X_\gamma\longrightarrow
                 \CP^{n_1+n_2+1}$,
                 which we call the primitive projective quotient.
                 
       \begin{mainthmE}[Projective Delaunay construction]
For the connected minimal Lagrangian immersions $\iota_i$ above,
every regular oscillatory contact profile $\gamma$ determines a Delaunay-type minimal Lagrangian immersion
$ \iota_\gamma$ in $\mathbb{CP}^{n_1+n_2+1}$.
If the $\iota_i$ are compact embeddings
and $\gamma$ is ordinarily closed,
then the spiral minimal product of their horizontal lifts
has an embedded primitive spherical quotient,
and $\overline \iota_\gamma$  is embedded precisely when
the reduced relative winding number $A$ is one.
Under these compactness and closing assumptions,
$
\operatorname{Ind}_{\mathrm{Ham}}(\overline \iota_\gamma)
\geq 
\operatorname{Ind}_{\mathrm{Ham}}(\iota_1)
+\operatorname{Ind}_{\mathrm{Ham}}(\iota_2)
+2N-2.
$
       \end{mainthmE}

          Here ``Delaunay-type'' means periodic neck--bulge variation, 
          not the  classical surfaces of Delaunay \cite{Delaunay1841}. 
          Winding number $A$ is defined in \cref{sub62}.
          The index is the
          %intrinsic exact-form 
          Hamiltonian index on the immersion domain; see
          \cref{ss9,co9}.
        %  For embedded inputs;  e.g. see \cref{tb1}.

%%%%%%%%%%%%%%%%%%%%%%%%%%%%%%%%%%

%\majorsectiongap
\section{Extrinsic Geometry and the Reduced Minimality Equations}
          \label{se2}

         In this section
            we separate the factor-normal and profile-normal parts of the minimality equation.
           This splitting proves Theorem A in \cref{th1}.
The calculation also leads to the curvature estimates used later.

   \subsection{Local regimes of spiral minimal products}
          \label{ss1}

        Retain the notation of \cref{df2} and \eqref{eq3}--\eqref{eq4}, 
        use the real Euclidean inner product,
        and write $g_i=f_i^*g_{\Sphere^{2n_i+1}}$.
       
       \begin{convention}
           \label{cv1}
           Statements involving phase coordinates or division by $a$ or $b$
            are understood componentwise on connected subintervals
            where $z_1z_2\ne0$.
           A prime denotes differentiation in the displayed parameter, and in an
           $s$-graph it means $\frac{d}{ds}$.
           From \cref{se4} onward,
           $t$ is affine for $\bar g$
           and a dot means $\frac{d}{dt}$.
           Round arclength derivatives are written $\frac{d}{d\tau}$.
       \end{convention}

  %     \medskip
       \noindent\textit{Profile notation and terminology.}
                   On a connected interval on which $z_1z_2\ne0$,
           the magnitudes are \emph{steady} when $a,b$ are constant
            and 
            \emph{nonsteady} otherwise.
            The branch is \emph{doubly spiral} when $\theta_1'\theta_2'\ne0$ throughout,
            and \emph{singly spiral} when exactly one phase is constant.
          For affinely parametrized weighted geodesics,
           the doubly and singly spiral terminology agrees respectively with
           $\mu_1\mu_2\ne0$ and with exactly one of $\mu_1,\mu_2$ vanishing.
          At either coordinate circle, where $a=0$ or $b=0$,
           we shall use ambient or signed coordinates.

          On such an interval,
           both phases are constant if and only if
           $(\theta_1',\theta_2')\equiv(0,0)$,
           and we call such a profile a \emph{meridian}.
          If an affinely parametrized weighted geodesic satisfies
           $(\theta_1',\theta_2')=(0,0)$ at one point,
           conservation forces $(\theta_1',\theta_2')\equiv(0,0)$,
           and hence it must be a meridian.
          Finally,
           if $a(t_0)=0$ with $k_1>0$,
           or $b(t_0)=0$ with $k_2>0$,
           the corresponding factor block of $dG_\gamma$ vanishes,
           so $G_\gamma$ is not an immersion there.

   \subsection{Consequences of \texorpdfstring{$\cC$}{C}-total reality}
          \label{ss2}

          Two useful properties of  $\cC$-total reality separate the factor-normal and profile-normal directions.

       \begin{lemma}%[Two consequences of $\cC$-total reality]
           \label{lm2}
           If $f:M^k\to\Sphere^{2n+1}$ is $\cC$-totally real,
            then
           \begin{equation}\label{eq9}
                              \langle Jdf(X),
                              df(Y)\rangle=0
                               \qquad\text{for all }X,
                              Y\in TM.
           \end{equation}
           Its spherical second fundamental form also has no Reeb component:
           \begin{equation}\label{eq10}
                              \langle B^f(X,
                              Y),
                              Jf\rangle=0
                               \qquad\text{for all }X,
                              Y\in TM.
           \end{equation}
           In particular,
            the spherical mean-curvature vector of $f$ is orthogonal to $Jf$.
       \end{lemma}

       \begin{proof}
           Since $f^*\alpha=0$ and $d\alpha(U,V)=2\langle JU,V\rangle$, the first identity follows.
            By differentiating $\langle df(Y),Jf\rangle=0$
            and using \eqref{eq9} in the spherical Gauss formula,
            we obtain
             \eqref{eq10}.
       \end{proof}

          For either factor $f$,
           define the reduced spherical normal bundle 
                   $\cN_f^\circ:=\{\xi\in\cN_f^{\Sphere}:\langle\xi,Jf\rangle=0\}$.
          By \cref{lm2},
           $B^f$ and the mean-curvature vector of $f$ take values in $\cN_f^\circ$.

          The same contact orthogonality gives the {induced} metric and the immersion criterion.

       \begin{lemma}[Induced metric and regularity]
           \label{lm1}
           The tangent blocks of $G_\gamma$ are
            mutually orthogonal and
           \begin{equation}\label{eq7}
                              G_\gamma^*g_{\Sphere}
                               =v_\gamma^2\,\dd t^2+a^2g_1+b^2g_2,
                               \qquad
                              v_\gamma^2=|\gamma'|_{\Sphere^3}^2.
           \end{equation}
           The map $G_\gamma$ is an immersion at $(t,x,y)$ precisely
            when $\gamma$ is immersed at $t$,
            both factors are immersions,
            with $a(t)>0$ if $k_1>0$
            and $b(t)>0$ if $k_2>0$,
            and 
            %its volume density is
           \begin{equation}\label{eq8}
                              \dd\vol_{G_\gamma}
                               =a^{k_1}b^{k_2}v_\gamma\,\dd t\,
                              \dd\vol_{g_1}\,\dd\vol_{g_2}.
           \end{equation}
       \end{lemma}

       \begin{proof}
           The factor blocks are orthogonal and have metrics $a^2g_1$ and
           $b^2g_2$.
           Their pairings with $G_\gamma'$ vanish by \cref{df1}.
           This proves \eqref{eq7}.
           A positive-dimensional factor block is nondegenerate 
           exactly when
            its coefficient is nonzero,
            which gives the immersion criterion.
            Correspondingly,  \eqref{eq8} follows.
       \end{proof}

   \subsection{Orthogonal tangent and normal splittings}
          \label{ss3}

          With $v_\gamma$ as in \eqref{eq7}, 
          define
            $E_0:=\frac{G_\gamma'}{v_\gamma}$.
          If $\{e_\alpha\}_{\alpha=1}^{k_1}$ and $\{u_\beta\}_{\beta=1}^{k_2}$ are local orthonormal frames on the two factors,
            define
          \[
                              E_\alpha=(e^{i\theta_1}df_1(e_\alpha),
                              0),
                               \qquad
                              F_\beta=(0,
                              e^{i\theta_2}df_2(u_\beta)).
          \]
          These are unit tangent fields of $G_\gamma$;
           the corresponding domain fields are $a^{-1}e_\alpha$
            and $b^{-1}u_\beta$.

          For $i=1,2$,
           let $\cN_i^\circ$ denote the natural lift of $\cN_{f_i}^\circ$ to the target product sphere.
           Let $\cN_{\mathrm{prof}}$ denote the natural lift of
                 $(\RR\gamma')^\perp\subset T\Sphere^3$ 
             under the real-linear isometry
           $(w_1,w_2)\mapsto(w_1f_1(x),w_2f_2(y))$.
          The ambient block orthogonality and a rank count show 
          that their orthogonal complement in the spherical normal bundle of $G_\gamma$ agrees with $\cN_{\mathrm{prof}}$,
            and give
          \[
                              TG_\gamma
                               =\RR E_0\oplus\Span\{E_\alpha\}
                              \oplus\Span\{F_\beta\},
                               \qquad
                              \cN_{G_\gamma}^{\Sphere}
                               =\cN_1^\circ\oplus\cN_2^\circ
                              \oplus\cN_{\mathrm{prof}}.
          \]

   \subsection{Second fundamental form}
          \label{ss4}

          Let $B$ denote the second fundamental form of $G_\gamma$ in the target sphere,
          $B_i$ 
          that of $f_i$ in each source sphere,
          and 
    $\mathbf H_i=\tr B_i$  the unnormalized  mean-curvature vector of $f_i$.

          For a vector $\xi_i$ along $f_i$, denote its phase-rotated lifts by
           $\xi_1^\sharp=(e^{i\theta_1}\xi_1,0)$
            and $\xi_2^\sharp=(0,e^{i\theta_2}\xi_2)$.
          The same notation is used for $B_i^\sharp$
            and $\mathbf H_i^\sharp$.
          Let $T_\gamma=\frac{\gamma'}{v_\gamma}$
          stand for the unit tangent to the profile,
           and set
          $
                              S=(-be^{i\theta_1},
                              ae^{i\theta_2})\in T_\gamma\Sphere^3,
                             %  \qquad
                              S^\perp=S-\langle S,
                              T_\gamma\rangle T_\gamma.
          $
          The same symbols $T_\gamma,S,S^\perp$ denote their natural lifts
           along $G_\gamma$.

       \begin{lemma}[Block identities of $B$]
           \label{lm4}
           The second fundamental form of $G_\gamma$ is given by
           \begin{align}
                              B(E_\alpha,
                              E_{\alpha'})
                               &=\frac1aB_1^\sharp(e_\alpha,
                              e_{\alpha'})
                               +\delta_{\alpha\alpha'}\frac ba S^\perp,
                              \label{eq11}\\
                              B(F_\beta,
                              F_{\beta'})
                               &=\frac1bB_2^\sharp(u_\beta,
                              u_{\beta'})
                               -\delta_{\beta\beta'}\frac ab S^\perp,
                              \label{eq12}\\
                              B(E_\alpha,F_\beta)&=0.
                              \label{eq13}
           \end{align}
           The remaining blocks are
           \[
            \begin{aligned}
                              B(E_0,E_\alpha)&=\frac{\theta_1'}{v_\gamma}JE_\alpha,
                               \quad
                              B(E_0,F_\beta)=\frac{\theta_2'}{v_\gamma}JF_\beta,
                               \quad
                              B(E_0,E_0)=\nabla^{\Sphere^3}_{T_\gamma}T_\gamma
                              \in\cN_{\mathrm{prof}}.
            \end{aligned}
           \]
       \end{lemma}

       \begin{proof}
                         At $p=(t,x,y)$ choose normal-coordinate frames on the factors centered at $x$ and $y$ respectively. 
                         Since $B$ is tensorial, 
                            it suffices to compute at $p$. 
                           Identities except the last follow by
                             spherical Gauss formula and the $\cC$-total reality of $f_1, f_2$. 
                            With fixed $(x,y)$, 
                            the profile lies in a great $\mathbb S^3$
                            of the target sphere,
                            which gives the last identity.
       \end{proof}

       \begin{proposition}[Mean-curvature splitting]
           \label{pr1}
           Let $\mathbf H^{G_\gamma}:=\tr B$ be the unnormalized mean-curvature vector of $G_\gamma$.
           It decomposes orthogonally as
           \begin{equation}\label{eq14}
                              \mathbf H^{G_\gamma}
                               =\frac1a\mathbf H_1^\sharp+\frac1b\mathbf H_2^\sharp
                               +\mathbf H_{\mathrm{prof}},
           \end{equation}
           where
           $\mathbf H_{\mathrm{prof}}
             :=\nabla^{\Sphere^3}_{T_\gamma}T_\gamma
               +\left(\frac{k_1b}{a}-\frac{k_2a}{b}\right)S^\perp$
            is the profile-normal component.
       \end{proposition}

       \begin{proof}
           This is the trace version of \cref{lm4}.
       \end{proof}

       \begin{theorem}[Weighted-geodesic characterization]
           \label{th1}
For the spherical $\cC$-totally real inputs $f_1, f_2$ above, 
         an immersed $G_\gamma$ is minimal 
             if and only if 
             both factors are minimal and $\gamma$ is an unparametrized geodesic of $\bar g$ on $I$,
       where     $\bar g=a^{2k_1}b^{2k_2} g_{\Sphere^3}$.
       \end{theorem}
       \begin{proof}
           For $\bar g=\rho^2g_{\Sphere^3}$,
           the standard conformal-change formula \cite[Theorem 1.159(a)]{Besse1987} gives
           \[
                              \nabla^{\bar g}_X Y
                               =
                               \nabla^{\Sphere^3}_X Y
                                +d(\log\rho)(X)Y+d(\log\rho)(Y)X
                                                                               -g_{\Sphere^3}(X, Y)
                                                                                                      \,            \nabla^{\Sphere^3}\log\rho.
           \]
                  Since $T_\gamma$ is unit for $g_{\Sphere^3}$ and conformal scaling preserves orthogonality,
            the normal component %is
           \[
                              (\nabla^{\bar g}_{T_\gamma}T_\gamma)^\perp
                               =
                               {\nabla^{\Sphere^3}_{T_\gamma}T_\gamma}
                               -\bigl(
                                               \nabla^{\Sphere^3}\log\rho
                                               \bigr)^\perp.
           \]
           The unparametrized geodesic equation is
           \begin{equation}\label{eq17}
                              \nabla^{\Sphere^3}_{T_\gamma}T_\gamma
                               =\bigl(\nabla^{\Sphere^3}\log\rho\bigr)^\perp        .
           \end{equation}
           We have
            $\nabla^{\Sphere^3}\log\rho=
            \left(-\frac{k_1b}{a}+\frac{k_2a}{b}\right)S$,
            so \eqref{eq17} agrees exactly with the vanishing of the profile
            component in \eqref{eq14}.
           By the orthogonal splitting \eqref{eq14}, the other components vanish exactly when the factors are minimal.
           At a coordinate-circle point compatible with immersion,
            the corresponding $k_i$ is zero;
            the same conclusion follows in the signed coordinate of \cref{rm1},
            with an identically vanishing block omitted.
       \end{proof}

   \begin{remark}[Coordinate-circle points]
           If $\gamma(t_0)$ lies on $\{z_i=0\}$ and $G_\gamma$ is  immersed at $t_0$, 
           then $k_i=0$ ,
               $\bar g$ is smooth and nondegenerate near $\gamma(t_0)$,
           and ``geodesic of $\bar g$'' has its usual meaning there.
           The coordinate-circle and signed-meridian cases are therefore
           included in \cref{th1} without a separate boundary equation.
       \end{remark}

       \begin{remark}
Although the characterization also follows from Takahashi's theorem
             \cite{Takahashi1966,ZhangTakahashi2025},
            we retain the direct mean-curvature splitting because its
            separate normal components
            will be
                      used in the
            curvature and index analysis.
       \end{remark}

   \subsection{Curvature decomposition and estimates}
          \label{ss5}

          Let $\tau$ denote round arclength.
         Then
           $T_\gamma=\frac{d\gamma}{d\tau}$.
          Set
          \[
                              \omega_i=\frac{d\theta_i}{d\tau},
                               \qquad
                              \kappa_\gamma=
                              \nabla^{\Sphere^3}_{T_\gamma}T_\gamma,
                               \qquad
                              \chi_\gamma=|S^\perp|^2.
          \]
          On a magnitude chart $a=\cos s$,
           $b=\sin s$ one
           has
           $\chi_\gamma=1-\left(\frac{ds}{d\tau}\right)^2
           =a^2\omega_1^2+b^2\omega_2^2$.

       \begin{proposition}[Exact curvature decomposition]
           \label{pr2}
           \begin{align}
                              |B^{G_\gamma}|^2
                               ={}&\frac{|B_1|^2}{a^2}
                               +\frac{|B_2|^2}{b^2}
                               +\left(k_1\frac{b^2}{a^2}
                               +k_2\frac{a^2}{b^2}\right)\chi_\gamma
                               +2k_1\omega_1^2
                               +2k_2\omega_2^2
                               +|\kappa_\gamma|^2.
                              \label{eq18}
           \end{align}
           If $k_i=0$,
            omit the corresponding tangent block
            and its terms.
           If $G_\gamma$ is minimal,
            then $\kappa_\gamma=(-k_1\tan s+k_2\cot s)S^\perp$,
            and hence
            $ |B^{G_\gamma}|^2$ becomes
           \begin{equation}
                           \frac{|B_1|^2}{\cos^2s}
                               +\frac{|B_2|^2}{\sin^2s}
                               +2k_1\omega_1^2
                               +2k_2\omega_2^2
                               +\Bigl[
                              k_1\tan^2s
                               +k_2\cot^2s
                               +\bigl(k_2\cot s-k_1\tan s\bigr)^2
                              \Bigr]\chi_\gamma.
                              \label{eq19}
           \end{equation}
       \end{proposition}

       \begin{proof}
       By squaring the block identities in \cref{lm4}
       and using the orthogonal normal decomposition 
      we get  \eqref{eq18}.
       When $G_\gamma$ is minimal,
       \cref{th1} says
$
\kappa_\gamma=(k_2\cot s-k_1\tan s)S^\perp
$
and the substitution consequently leads to \eqref{eq19}.
                \end{proof}

\begin{corollary}[Uniform curvature bound]
    \label{co1}
    Assume that $G_\gamma$ is minimal,        $K_i:=\sup_{M_i}|B_i|<\infty $
    and
    $0<s_-\leq  s\leq  s_+<\frac{\pi}{2}$
    along $\gamma$.
    Then, pointwise,
    \begin{equation}      \label{eq20}
        |B^{G_\gamma}|^2
        \leq {}
        \frac{K_1^2+2k_1}{\cos^2s_+}
        +\frac{K_2^2+2k_2}{\sin^2s_-}
        +k_1\tan^2s_+
        +k_2\cot^2s_-  
        +\bigl(k_1\tan s_+ + k_2\cot s_-\bigr)^2 .
    \end{equation}
\end{corollary}

\begin{proof}
    As 
    $\chi_\gamma
                       =a^2\omega_1^2+b^2\omega_2^2\leq 1$,
    we get $\omega_1^2\leq  a^{-2}$ and $\omega_2^2\leq  b^{-2}$.
 With $ |k_2\cot s-k_1\tan s|        \leq       k_2\cot s_-        +k_1\tan s_+$
     and the monotonicity of $\tan$ and $\cot$, 
    \eqref{eq19} gives \eqref{eq20}.
\end{proof}

%\majorsectiongap
\section{\texorpdfstring{Weighted Geodesics, First Integrals, and Turning Points}
 {Weighted Geodesics, First Integrals, and Turning Points}}
          \label{se3}

          We integrate the weighted-geodesic equation, determine its admissible levels, and continue solutions across the turning points.

   \subsection{Noether conservation laws on a magnitude graph}
          Consider a nonsteady profile on a monotone magnitude interval.
          Using  $s$ as parameter, 
          the profile can be written as
          $
                              \gamma(s)=
                              \bigl(
                              a(s)e^{i\theta_1(s)},
                              b(s)e^{i\theta_2(s)}
                              \bigr)
                              \text{where }
                              a=\cos s
                              \text{ and }
                              b=\sin s.
        $
         We shall use prime to denote $\frac{d}{ds}$
         on the $s$-graph as in \cref{cv1}.
          Set $\Theta(s)=a^2(\theta_1')^2+b^2(\theta_2')^2$
            and $\rho=a^{k_1}b^{k_2}$,
             so that
             $\left\lvert\frac{d\gamma}{ds}\right\rvert_{g_{\Sphere^3}}^2
             =1+\Theta(s)$.

          The weighted length of an $s$-graph is $\int\mathcal L\,ds$,
            where $\mathcal L=\rho(s)\sqrt{1+\Theta}$.
          Nearby unparametrized curves may again be represented as $s$-graphs,
            and the two phase variables are cyclic.
          The Euler--Lagrange equations of this weighted-length functional
          are therefore
          \begin{equation}\label{eq21}
                              \frac d{ds}\left(\frac{\rho a^2\theta_1'}{\sqrt{1+\Theta}}\right)=0,
                               \qquad
                              \frac d{ds}\left(\frac{\rho b^2\theta_2'}{\sqrt{1+\Theta}}\right)=0.
          \end{equation}
   \subsection{First integrals and the explicit phase pair}
                                   On a monotone $s$-branch
                                   in the chart $\mu_1\ne0$,
                                   let $t$ be an affine parameter
                                   and let dots denote differentiation with respect to $t$.
                                   Since
                                   $    \sqrt{2\cH}=\rho|\dot s|\sqrt{1+\Theta} $,
                                   the two conserved quantities in \eqref{eq21} satisfy
$$
    \frac{\rho a^2\theta_1'}{\sqrt{1+\Theta}}
    =\sgn(\dot s)\frac{\mu_1}{\sqrt{2\cH}}
    =\frac{\sigma}{\sqrt{C_2}},
    \qquad
    \frac{\rho b^2\theta_2'}{\sqrt{1+\Theta}}
    =\sgn(\dot s)\frac{\mu_2}{\sqrt{2\cH}}
    =\frac{\sigma C_1}{\sqrt{C_2}},
$$
where
           $ 
                 \sigma:=\sgn(\dot s)\sgn(\mu_1)\in\{\pm1\}.
            $
          The sign $\sigma$ encodes the direction of the magnitude branch
           relative to the fixed sign of $\mu_1$ and changes sign across a turning point,
           whereas $C_1,C_2$ are branch-independent.
          Eliminating common normalization leads to
           $b^2\theta_2'=C_1a^2\theta_1'$,
            and 
            moreover,
            $\frac{\Theta}{1+\Theta}
            =\frac{b^2+C_1^2a^2}{C_2a^{2k_1+2}b^{2k_2+2}}$.
        Set 
                $$
                         P_{C_1}(s):=    \frac{b^2+C_1^2a^2}{\rho^2a^2b^2},
                         % \text{ and }
                          \qquad
                          D_{C_1,C_2}(s)
                              :=
                              \rho^2a^2b^2\bigl(C_2-P_{C_1}(s)\bigr)
                              =
                              C_2a^{2k_1+2}b^{2k_2+2} -(b^2+C_1^2a^2).
                        $$
          The corresponding threshold function is
          \begin{equation}\label{eq24}
                              R_{k_1,k_2}(C_1)
                              :=\inf_{0<s<\frac{\pi}{2}}P_{C_1}(s),
          \end{equation}
          where the boundary value at $s=0$ is included by continuity
            for $(C_1,k_2)=(0,0)$.
          For a given branch,
           its magnitude interval is a connected component of
           $\left\{s\in\left(0,\frac{\pi}{2}\right):
           D_{C_1,C_2}(s)>0\right\}$.
          On this interval the reduced phase equations are
          \begin{equation}\label{eq25}
          % \begin{aligned}
                               \theta_1' =\frac{\sigma\tan s}{\sqrt{D_{C_1,C_2}(s)}},
                               \qquad
                                \theta_2'=\frac{\sigma C_1\cot s}{\sqrt{D_{C_1,C_2}(s)}}.
         %  \end{aligned}
          \end{equation}
          The levels $C_1=0$ and $C_1=-1$ correspond respectively to the
           singly spiral and contact cases.

   \subsection{\texorpdfstring{Threshold function and simple turning points}
    {Threshold function and simple turning points}}

       \begin{proposition}[Threshold minimum and turning points]
           \label{pr3}
           If $(C_1,k_2)\neq(0,0)$, 
                            then $P_{C_1}$ diverges at both endpoints and has a unique critical point
                                              $s_*\in\left(0,\frac{\pi}{2}\right)$,
                                              which is a nondegenerate global minimum.
        For every $C_2>R_{k_1,k_2}(C_1)$,             
                $D_{C_1,C_2}$ has exactly two simple zeros      $s_-<s_+$,
                 and
    $  \{D_{C_1,C_2}>0\}=(s_-,s_+) $.
    If $C_1=k_2=0$, then
    $
        P_0(s)=\cos^{-2k_1-2}s,
        R_{k_1,0}(0)=P_0(0)=1.
   $
    In this case $P_0$ is strictly increasing, and every
    $C_2>1$ gives a unique interior zero $s_+$ of
    $D_{0,C_2}$, with
$
        \{D_{0,C_2}>0\}=(0,s_+).
$
       \end{proposition}
       \begin{proof}
           Set $x=\cos^2s$.  
           When $C_1\neq 0$,
           the critical-point equation for
           $P_{C_1}=\frac{1+(C_1^2-1)x}{x^{k_1+1}(1-x)^{k_2+1}}$
             is    
                     $(C_1^2-1)(k_1+k_2+1)x^2      +(2k_1+k_2+2-C_1^2k_1)x       -(k_1+1)=0$
           which has endpoint values
            $-(k_1+1)$ and $C_1^2(k_2+1)$.
           No matter whether $C_1^2=1$ or not,
           there is exactly one root in $(0,1)$
           which is simple and decides the global minimum of $P_{C_1}$. 
           For $C_1=0$,
          $P_0$ becomes 
           $x^{-k_1-1}(1-x)^{-k_2}$.
                       If $k_2>0$, 
                                   its derivative has a unique simple zero in $(0,1)$    which gives the interior minimum. 
                        If $k_2=0$, 
                                then
                                $    P_0(s)=\cos^{-2k_1-2}s $
                                        is strictly increasing from its boundary minimum $P_0(0)=1$.
                                Finally, at any interior zero $s_0$ of $D_{C_1,C_2}$,
                                $    D_{C_1,C_2}'(s_0)
                                      =-\rho^2a^2b^2P_{C_1}'(s_0)\ne0,
                                   $
                                   so the zero is simple.
       \end{proof}

  \subsection{Global continuation and magnitude oscillation}
          \label{ss6}

       \begin{proposition}[Geometric continuation and global magnitude oscillation]
           \label{pr4}
           Fix         a level  $C_2>R_{k_1,k_2}(C_1)$
           having two simple turning values
           $0<s_-<s_+<\frac{\pi}{2}$.
                    Every profile on this level admits a unique global real-analytic continuation, 
           whose magnitude is periodic between $s_-$ and $s_+$.
       \end{proposition}

       \begin{proof}
                The energy identity 
                $\rho(s)^2\dot s^{\,2} =\mu_1^2(C_2-P_{C_1}(s))$
                  confines the magnitude to $[s_-,s_+]$.
                On this compact band,
                the metric $\bar g$ is real analytic and nondegenerate,
                while conservation of energy bounds the velocity.
                Hence the given geodesic extends uniquely as a real-analytic geodesic for all affine time.

      Since the zeros of $C_2-P_{C_1}$ at $s_\pm$ are simple,
      the energy identity shows that each passage between them takes finite time.
         Through differentiating the energy identity where $\dot s\neq0$ and extending by continuity,
         we obtain
         \begin{equation}\label{eqsode}
              \ddot s  +  \frac{\rho'(s)}{\rho(s)}\dot s^{\,2}
               =
                   -\frac{\mu_1^2P_{C_1}'(s)}{2\rho(s)^2}.
          \end{equation}
              At either turning point,
              $\dot s=0$ and $P_{C_1}'(s)\neq0$,
              so \eqref{eqsode} gives $\ddot s\neq0$.
              Thus the magnitude reverses direction there.
              Time-reversal symmetry and uniqueness for \eqref{eqsode}
              then show that successive passages are reflections of one another.
              Hence the magnitude oscillates periodically between $s_-$ and $s_+$.
       \end{proof}

      \begin{remark}[{Passage through a coordinate circle}]
           \label{rm1}
           By symmetry, 
                 it is enough to consider a regular level reaching $s=0$
                 which forces $C_1=k_2=0$. 
                 Since $\mu_2=0$, 
                 after a constant phase rotation we can write $z_2=u\in\RR$,
                 so that $b=|u|$.
                 Now the metric
                 $    \bar g=(1-|z_2|^2)^{k_1}g_{\Sphere^3}$
                 is real analytic and nondegenerate at $u=0$. 
                                   Moreover, for $C_2>1$,
                $
                    \left(\frac{du}{d\tau}\right)^2
                    \big|_{u=0}
                    =\frac{C_2-1}{C_2}>0.
                 $
                 Hence the geodesic equation gives a unique real-analytic continuation in the signed coordinate $u$.
                 In the polar coordinate $b=|u|$,
                 the same passage is represented by two polar branches meeting at $b=0$,
                 with the sign change by $\theta_2\mapsto\theta_2+\pi$.
          \end{remark}

          \par
       \begin{remark}[Geometric viewpoint on global spiral minimal products]
           \label{rmhl}
           The extension theorem of Harvey and Lawson \cite{HarveyLawson1975}
           gives an independent global formulation at the level of the image:
           adjacent product branches meeting at an ordinary turn or a coordinate-circle passage as described in \cref{rm1} form locally a single real-analytic minimal submanifold.
       \end{remark}

   \subsection{The threshold curve}
          \label{ss7}

          Abbreviate
          $
                              L=k_1+k_2+1.
        $
                      For $t\geq 0$, set
                      $$
                            F_{k_1,k_2}(x,t)
                         =
                            \frac{1+(t-1)x} {x^{k_1+1}(1-x)^{k_2+1}},
                            \     \qquad 0<x<1.
                           $$
                        Then, with $x=\cos^2s$, 
                        \eqref{eq24} gives
          \begin{equation}\label{eq29}
                          R_{k_1,k_2}(C_1)
                           =
                           \inf_{0<x<1}F_{k_1,k_2}(x,C_1^2)
                           =
                           \inf_{0<s<\frac{\pi}{2}}P_{C_1}(s).
          \end{equation}

       \begin{proposition}[Admissible parameter domain]
           \label{th3}
           The function $R_{k_1,k_2}$ is even and continuous.  
           %It is real analytic when $(C_1, k_2)\neq (0,0)$. %$C_1\ne0$, and also at $C_1=0$ when $k_2>0$.  
            Moreover, the set            \begin{equation}\label{eq30}
                              \mathcal U_{k_1,k_2}
                              :=\{(C_1,C_2)\in\RR^2:C_2>R_{k_1,k_2}(C_1)\}
            \end{equation}
            is an open connected parameter domain. 
                   \end{proposition}

       \begin{proof}
    This is clear.  
           \end{proof}

       \begin{remark}[Factor exchange]\label{rm2}
       For $C_1\neq0$, 
       interchanging   factors induces
       $
        (k_1,k_2;C_1,C_2)
        \mapsto
        \left(k_2,k_1;C_1^{-1},\frac{C_2}{C_1^2}\right).
       $
        Consequently,
        $
        R_{k_1,k_2}(C_1)
        =
        C_1^2R_{k_2,k_1}(C_1^{-1}),
        $
        so the behavior of the threshold at infinity reduces to that near
        $C_1=0$, with $k_1$ and $k_2$ interchanged.
        \end{remark}

%\majorsectiongap
\section{Hamiltonian Reduction, Phase Maps, and Global Closing}
          \label{se4}

        In this section, 
          we use Hamiltonian reduction to study the global closing problem for profiles. 
          Fixing the conserved phase momenta reduces the dynamics to a one-degree-of-freedom system. 
          Its regular nonsteady energy levels give the oscillatory profiles, 
          while its equilibria give the steady profiles. 
          The corresponding Routhian yields the scalar variational problem used in Section \ref{se9}.

    The phase advances over one complete magnitude period define the complete-cell phase map, 
    which is real analytic. 
    This gives the exact ordinary closing criterion 
    and 
    shows that every open subset of the admissible region contains ordinarily closed profiles with arbitrarily large return order. 
    Finally, finite phase stabilizers of the factors account for spherical returns 
    and 
    the corresponding primitive spherical quotients.

   \subsection{Hamiltonian and Routh reduction}

          Let $Q=\left(0,\frac{\pi}{2}\right)\times\TT^2$
          with coordinates $(s,\theta_1,\theta_2)$.
          With $                               
                              W_0(s)=\rho(s)^2,\,
                               W_1(s)=\rho(s)^2\cos^2s,\,
                               W_2(s)=\rho(s)^2\sin^2s$,
          $
                               \bar g=W_0\,\dd s^2+W_1\,\dd\theta_1^2+W_2\,\dd\theta_2^2,
          $
         the weighted energy density with respect to affine parameter $t$ is
          \[
                              \cL_E
                               =\frac12\left(W_0\dot s^{2}
                               +W_1\dot\theta_1^{2}
                               +W_2\dot\theta_2^{2}\right).
          \]
          Denote round and weighted arclength by $\tau$ and $\ell_{\bar g}$.
          Then along a trajectory of energy $E=\cH$,
          \begin{equation}\label{eq34}
                              \dd\ell_{\bar g}=\rho(s)\,\dd\tau=\sqrt{2E}\,\dd t.
          \end{equation}
           The phase translations form an isometric $\TT^2$-action,
          with momentum map $\bmu=(\mu_1,\mu_2)$:
          \begin{equation}\label{eq35}
                              p_s=W_0\dot s,
                               \qquad
                               \mu_1=W_1\dot\theta_1,
                               \qquad
                               \mu_2=W_2\dot\theta_2.
          \end{equation}
          The Legendre transform gives the geodesic Hamiltonian
          \begin{equation}\label{eq36}
                              \cH(s ;p_s,\bmu)
                               =\frac12\left(
                              \frac{p_s^2}{W_0}
                               +
                              \frac{\mu_1^2}{W_1}
                               +
                              \frac{\mu_2^2}{W_2}
                              \right).
          \end{equation}
          Since $\theta_1$ and $\theta_2$ are cyclic,
          $\mu_1$ and $\mu_2$ are conserved.
          For fixed $\bmu$, set
          \begin{equation}\label{eq40}
           U_{\bmu}(s)
                               =
                              \frac12\sum_{i=1}^2
                              \frac{\mu_i^2}{W_i(s)}.
          \end{equation}
          Then the reduced Hamiltonian is
          $
                              \cH_{\bmu}(s,p_s)
                               =
                              \frac{p_s^2}{2W_0(s)}
                               +
                              U_{\bmu}(s).
          $
        
       The partial Legendre transform in the cyclic velocities gives the Routhian \cite{Routh1877}.
       For a modern geometric formulation, 
       see \cite[Sections III.B and III.D]{MarsdenRatiuScheurle2000}.
          \[
                              \mathscr R_{\bmu}
                               :=\left.\left(\cL_E-\sum_{i=1}^2\mu_i\dot\theta_i\right)
                              \right|_{\dot\theta_i=\frac{\mu_i}{W_i}}
                               =\frac12W_0(s)\dot s^{\,2}-U_{\bmu}(s).
          \]
          Reconstruction is given by
          $\dot\theta_i=\frac{\mu_i}{W_i(s)}$,
          and the Routh energy is
          \[
                              \cE_{\mathscr R_{\bmu}}
                               =
                              \frac12W_0(s)\dot s^2+U_{\bmu}(s)
                               =
                              \cH_{\bmu}(s,W_0\dot s)
                               =
                              E.
          \]
          The scalar operator in \cref{se9} is the Hessian of the Routh action,
          and its relation to the full profile Hessian is given by \cref{th7}.

\subsection{Steady profiles}

               Assume $(C_1,k_2)\ne(0,0)$.
          A profile is steady exactly when
          $                              C_2=R_{k_1,k_2}(C_1)$.
          In this case $s(t)\equiv s_*$,
          where $s_*$ is the unique minimizer of $P_{C_1}$.
          Setting $a_*=\cos s_*$ and $b_*=\sin s_*$,
          phase translation and affine reparametrization give
          \[
                              \gamma_*(t)=(a_*e^{it},b_*e^{i\hat c t}),
                              \qquad \hat c=\frac{\dot\theta_2}{\dot\theta_1},
                              \qquad
                              \frac{k_1}{L}<a_*^2 \leq  \frac{k_1+1}{L},
                              \qquad
                              \hat c^{\,2}
                               =
                              \frac{a_*^2(k_1+1-La_*^2)}
                              {b_*^2(La_*^2-k_1)}.
          \]
          Here $\hat c=C_1\cot^2s_*$ and the profile is %ordinarily
          closed exactly when $\hat c\in\QQ$. 
          If $k_1=k_2=0$, then
          $\hat c^{\,2}=1$ and the profiles are great circles.

    \subsection{The complete-cell phase map}
          \label{se6}

  By \cref{pr3}, a profile in the chart $\mu_1\ne0$ is regular oscillatory exactly when
           $(C_1,C_2)\in\mathcal U_{k_1,k_2}$ if $k_2>0$, 
           and when
           $(C_1,C_2)\in\mathcal U_{k_1,0}\cap\{C_1\ne0\}$ if $k_2=0$.
     In either case, it is doubly spiral exactly when $C_1\ne0$.
     Thus the two regular doubly spiral chambers are
     $\mathcal U_{k_1,k_2}^{\pm}:=\mathcal U_{k_1,k_2}\cap\{\pm C_1>0\}$. 
    For such a profile, the energy equation gives
          $$
                              p_s^2
                               =\mu_1^2W_0(s)\bigl(C_2-P_{C_1}(s)\bigr)
          $$
          with turning values $s_-<s_+$,
          and
         the affine time of one complete cell is
            \[
                              T_{\mathrm{cell}}
                               =\frac{2}{|\mu_1|}\int_{s_-}^{s_+}
                              \frac{\rho(s)}{\sqrt{C_2-P_{C_1}(s)}}\dd s
                               =\frac{2}{|\mu_1|}\int_{s_-}^{s_+}
                              \frac{\rho(s)^2\sin s\cos s}{\sqrt{D_{C_1,C_2}(s)}}\dd s.
          \]

          Recall
          $\varepsilon_1=\sgn(\mu_1)$. 
          Integrating \eqref{eq25} over the two monotone halves of a cell gives
          \begin{equation}\label{eq47}
%           \begin{aligned}
                              \Delta_1
                               %&
                               =2\varepsilon_1\int_{s_-}^{s_+}
                                 \frac{\tan s}{\sqrt{D_{C_1,C_2}(s)}}\,\dd s,
                                 \qquad
                              \Delta_2
                               %&
                               =2\varepsilon_1C_1\int_{s_-}^{s_+}
                                 \frac{\cot s}{\sqrt{D_{C_1,C_2}(s)}}\,\dd s.
%           \end{aligned}
          \end{equation}

          These increments define the complete-cell phase map $\Pi$  in \eqref{eq6}. 
          When $k_2=0$,
          a continuous real-valued lift across $C_1=0$ requires a $2\pi$-adjustment.
          
      \begin{lemma}[Adjusted phase lift across $C_1=0$]
           \label{lm6}
           Assume $k_2=0$.
           For $C_1\ne0$,
           let $\Delta_i$ be the complete-cell increments in \eqref{eq47}.
           Along $C_1=0$ in the interior of $\mathcal U_{k_1,0}$, 
           define
           \[
             \Delta_1(0,C_2)
              =2\varepsilon_1\int_0^{s_+(C_2)}
                \frac{\dd s}{\cos s\sqrt{C_2\cos^{2(k_1+1)}s-1}},
             \qquad
             \cos s_+(C_2)=C_2^{-\frac{1}{2(k_1+1)}},
           \]
           and set
           \begin{equation}\label{eq42}
             \Delta_2^\sharp(C_1,C_2)=
             \begin{cases}
               \Delta_2(C_1,C_2),&C_1>0,\\
               \varepsilon_1\pi,&C_1=0,\\
               \Delta_2(C_1,C_2)+2\varepsilon_1\pi,&C_1<0.
             \end{cases}
           \end{equation}
            Then the adjusted phase map
          \[
             \Pi^\sharp(C_1,C_2)
              :=
             \frac{1}{2\pi}
             \bigl(\Delta_1(C_1,C_2),
                   \Delta_2^\sharp(C_1,C_2)\bigr)
          \]
          extends real analytically across $C_1=0$.
          For $C_1\ne0$,
          it represents the same torus-valued return as $\Pi$.
          The unadjusted lift has one-sided limits
          $\Delta_2\to\pm\varepsilon_1\pi$.
          If $k_1>0$, then
          \[
             \varepsilon_1\,
             \partial_{C_1}\Delta_2^\sharp(0,C_2)>0,
             \qquad C_2>1.
          \]
               \end{lemma}

       \begin{proof}
           See \cref{ap2}.
       \end{proof}

       \begin{proposition}[Analyticity of the phase map]\label{pr5}
      If $k_2>0$,
       then $\Pi$ is real analytic on $\mathcal U_{k_1,k_2}$.
       If $k_2=0$,
       then $\Pi^\sharp$ is real analytic on $\mathcal U_{k_1,0}$.
       \end{proposition}

\begin{proof}
As the geodesics under our consideration on the weighted sphere never meet any degenerate locus of the weighted metric,
the statement directly follows by the analytic dependence of geodesics on initial position and tangent direction.

Using pure analysis, one can instead argue as follows.
For $k_2>0$,
apply \cref{lm20} to \eqref{eq47}.
The case $k_2=0$ is \cref{lm6}.
\end{proof}

   \subsection{Closing and full rank}
          \label{se7}

       \begin{proposition}[Ordinary closing and return order]\label{th4}
           A regular oscillatory profile is ordinarily closed 
           if and only if $\Pi(C_1,C_2)\in\QQ^2$.
           When this holds,
          assume
            $\frac{\Delta_i}{2\pi}=\frac{\hat p_i}{\hat q_i}$,
            where $\hat q_i$ is a positive integer with $\gcd(\hat p_i,\hat q_i)=1$,
           ($\widehat q_i=1$ when $\Delta_i=0$).
           Then the ordinary return order is
            $m_\gamma=\lcm(\widehat q_1,\widehat q_2)$
            and the primitive ordinary period is $m_\gamma T_{\mathrm{cell}}$.
       \end{proposition}

       \begin{proof}
           Every ordinary period contains an integral number $m$ of complete cells. 
           Closing is equivalent to $m\Delta_i\in2\pi\ZZ$ for $i=1,2$, 
           whose least positive solution is $\lcm(\widehat q_1,\widehat q_2)$.
       \end{proof}

       \begin{lemma}[Endpoint limits at $C_1=0$]
           \label{lm8}
          Let $k_1\geq 0$ and $k_2\geq 0$.
          The first phase increment
         $C_2\mapsto\Delta_1(0,C_2)$
          %defined by \eqref{eq47}
           is real analytic for
          $C_2>R_{k_1,k_2}(0)$ and satisfies
           \begin{align}
                              |\Delta_1(0,C_2)|&\longrightarrow
                              \begin{gathered}
                                 \begin{cases}
                                    \displaystyle \frac\pi{\sqrt {k_1+1}},&k_2=0,\\[4pt]
                                    \displaystyle \pi\sqrt{\frac{2}{k_1+1}},&k_2>0,
                                 \end{cases}
                              \end{gathered}
                               &&
                                    C_2\downarrow R_{k_1,k_2}(0),
                              \label{eq48}\\
                              |\Delta_1(0,C_2)|&\longrightarrow
                              \frac\pi {k_1+1}
                               && C_2\uparrow\infty.
                               \label{eq49}
           \end{align}
         If $k_1+k_2>0$, then
           $\partial_{C_2}\Delta_1(0,C_2)\ne0$
           at some  $C_2>R_{k_1,k_2}(0)$.
       \end{lemma}

       \begin{proof}
           See \cref{ap4}.
       \end{proof}

           \begin{theorem}[Full rank and dense closing]
    \label{th5}
        Assume $k_1+k_2>0$.
       Then $\Pi$
                  (or $\Pi^\sharp$ when $k_2=0$)
            is real analytic on $\mathcal U_{k_1,k_2}$
        and has open dense full-rank locus.
    Consequently,
       ordinarily closed profiles with arbitrarily large
        primitive closing order are dense in $\mathcal U_{k_1,k_2}$.
\end{theorem}

       \begin{proof}
According to           \cref{pr5},
            $\Pi$  
                    (or $\Pi^\sharp$ when $k_2=0$)
                    is
            analytic on the connected domain
           $\mathcal U_{k_1,k_2}$. 
           If $k_2>0$, 
           then
           $D_{-C_1,C_2}=D_{C_1,C_2}$,
           so $\Delta_1$ is even and $\Delta_2$ is odd in $C_1$.
           Hence
           \begin{equation}\label{eq50}
             \begin{gathered}
                              \partial_{C_1}\Delta_1(0,C_2)=0,
                               \qquad \qquad\quad\ 
                              \partial_{C_2}\Delta_2(0,C_2)=0,\\
                              \partial_{C_1}\Delta_2(0,C_2)
                               =2\sgn(\mu_1)\int_{s_-}^{s_+}
                              \frac{\cot s}{\sqrt{D_{0,C_2}(s)}}\,\dd s\ne0.
             \end{gathered}
           \end{equation}
           Since
           $\Pi=\left(\frac{\Delta_1}{2\pi},
           \frac{\Delta_2}{2\pi}\right)$,
we obtain
           \begin{equation}\label{eq51}
                              \det D\Pi(0,C_2)
                               =-\frac1{(2\pi)^2}
                              \partial_{C_2}\Delta_1(0,C_2)
                              \partial_{C_1}\Delta_2(0,C_2).
           \end{equation}
 The distinct limits in \eqref{eq48}--\eqref{eq49}
           give some $C_2^*$ for which
           $\partial_{C_2}\Delta_1(0,C_2^*)\ne0$.
           Hence \eqref{eq51} is nonzero at $C_2^*$.

           If instead $k_2=0$,
           then $k_1>0$.
           Along $C_1=0$,
           the function $\Delta_1$ is even in $C_1$,
           while $\Delta_2^\sharp(0,C_2)=\varepsilon_1\pi$;
           hence the two diagonal derivatives vanish.
           By the distinct limits in \eqref{eq48}--\eqref{eq49},
           choose $C_2^*>1$ with
           $\partial_{C_2}\Delta_1(0,C_2^*)\ne0$.
           \cref{lm6} gives
           $\partial_{C_1}\Delta_2^\sharp(0,C_2^*)\ne0$.
           Therefore
           \[
              \det D\Pi^\sharp(0,C_2^*)
              =-\frac1{(2\pi)^2}
              \partial_{C_2}\Delta_1(0,C_2^*)
              \partial_{C_1}\Delta_2^\sharp(0,C_2^*)\ne0.
           \]
   %        
       %   If $k_2>0$,
%the transverse derivative is nonzero by \eqref{eq50}.
%If $k_2=0$,
%the corresponding statement for the adjusted lift follows from
%\cref{lm6}.
           So the nonvanishing locus is open and dense.
           The last claim simply follows.
       \end{proof}
  
\begin{remark}
      If $k_1=k_2=0$, then $\bar g$ is the standard round metric 
      and
       $\Pi=\frac12(\sgn\mu_1,\sgn\mu_2)$ on every regular nonsteady doubly spiral family. 
       Thus $D\Pi=0$, showing that the hypothesis in \cref{th5} is sharp.
\end{remark}

   \subsection{Phase stabilizers and spherical return}
          \label{se8}

         After $m$ complete cells, 
         the profile closes when  $e^{im\Delta_i}=1$ for $i=1,2$, 
         whereas the spiral product may return earlier through phase symmetries of the factors.

       \begin{definition}[Phase stabilizer]
           \label{df5}
          For a compact connected embedding
           $f:M\to\Sphere^{2n+1}$, 
           %define
           \begin{equation}\label{eq53}
             \mathsf H_f:=\{\zeta\in\Sphere^1:\zeta f(M)=f(M)\}.
           \end{equation}
           For $\zeta\in\mathsf H_f$, 
           embeddedness determines a unique $\varphi_\zeta\in\Diff(M)$ by
           $f\circ\varphi_\zeta=\zeta f$.
       \end{definition}

       \begin{lemma}[Finite stabilizer]
           \label{lm9}
              If $M$ is compact and $f$ is $\cC$-totally real,
    then $\mathsf H_f$ is finite cyclic,
    and its induced action on $(M,g_f)$ is faithful and isometric.
       \end{lemma}

       \begin{proof}
           The stabilizer is a closed subgroup of $\Sphere^1$. 
           If it were infinite, 
           it would be all of $\Sphere^1$; differentiating its pulled-back action would make $Jf$ tangent to $f(M)$,
           contrary to $\cC$-total reality.
            Hence it is finite cyclic, 
            and
           embeddedness together with unitary phase rotation gives the faithful isometric action.
       \end{proof}

      For simplicity, set $\mathsf H_i:=\mathsf H_{f_i}$.
       \begin{definition}[Spherical return data]
           \label{df6}
          Define $  
          \mathcal R_G
              :=\{m\in\ZZ:e^{im\Delta_i}\in\mathsf H_i,
                    \ i=1,2\}
                    $.
                    If $ \mathcal R_G
                    =m_G\ZZ\ne\{0\}
                    $
          with $m_G>0$, 
    set
           $T_G:=m_GT_{\mathrm{cell}}$,
           $h_i:=e^{im_G\Delta_i}\in\mathsf H_i$, and
           $\Phi:=\varphi_{h_1}\times\varphi_{h_2}$.
       \end{definition}

       \begin{proposition}[Spherical return]\label{th6}
           Let the factors be compact, connected, embedded, and   $\cC$-totally real, 
           and $\gamma$
            a regular oscillatory profile with $\mathcal R_G\ne\{0\}$. 
            Then $G_\gamma$ descends 
            %smoothly 
            to
           \[
             \cM_G:=([0,T_G]\times M_1\times M_2)/
             ((T_G,x,y)\sim(0,\Phi(x,y))).
           \]
           Moreover, $\gamma$ is ordinarily closed and
           $m_\gamma=m_G\ord(\Phi)$.
       \end{proposition}

       \begin{proof}
    By the definitions of $\Delta_i$, $m_G$, and $h_i$,
 $
        z_i(t+T_G)
        =
        e^{im_G\Delta_i}z_i(t)
        =
        h_i z_i(t),\,  i=1,2.
    $    
    Since $f_i\circ\varphi_{h_i}=h_i f_i$,
    it follows that
$
        G_\gamma(t+T_G,x,y)
        =
        G_\gamma(t,\Phi(x,y)),
$
    which proves the descent.
    By \cref{lm9},
    $\Phi$ has finite order.
    Faithfulness shows that an ordinary return after $rm_G$ cells
    occurs exactly when $\Phi^r=\Id$.
    Hence the least such $r$ is $\ord(\Phi)$.
       \end{proof}

%\majorsectiongap
\section{Second Variation and Morse Index}
          \label{se9}

          We separate the instability coming from the profile from that of the
           two factors.
          For a minimal immersion
          $f:M^k\to\Sphere^{n}$,
           define
          \[
             Q_f(\xi)=\int_M\bigl(|\nabla^\perp\xi|^2-|A_\xi|^2-k|\xi|^2\bigr)\,\dd\vol_M
          \]
          and $\Ind(f)$ its maximal negative dimension.
          We use the
          standard second variations of weighted energy and length, denoted
          $I_E$ and $I_L$; see \cite{Simons1968,Eastham1973,Duistermaat1976}.
          Along every profile considered below,
          $\bar g$ is smooth and nondegenerate;
          at a coordinate-circle passage,
          $G_\gamma$'s being immersion forces the corresponding $k_i$ to vanish,
          and the assertion follows from the signed coordinate of \cref{rm1}.
          Thus the usual geodesic index form has meaning
          and we use
          $\Ind_{\bar g}(\gamma)$ for its Morse index.

\subsection{Routh reduction}
          Fix the factors and let $\gamma$ be a closed minimal profile for
          which $G_\gamma$ is a smooth immersion.
          Parametrize $\gamma$
          affinely and set $c_\gamma=|\dot\gamma|_{\bar g}$.
          A
          $\bar g$-normal field $V=(V_1,V_2)$ along
          $\gamma\subset\Sphere^3$ induces the normal field in the target sphere
          \[
             \widetilde V(t,x,y)=(V_1(t)f_1(x),V_2(t)f_2(y)).
          \]

       \begin{lemma}[Profile variation formula]
           \label{pr8}
         Suppose the factors are compact minimal $\cC$-totally real  immersions.  
         For every periodic normal field $V$,
           \begin{equation}\label{eq62}
             Q_{G_\gamma}(\widetilde V)
              =\Vol(M_1)\Vol(M_2) I_L(V,V)
              =\frac{\Vol(M_1)\Vol(M_2)}{c_\gamma}I_E(V,V).
           \end{equation}
         The same formula holds on a compact profile segment 
         $[0,T]$ for fields with Dirichlet boundary conditions
         %with 
         and $Q_{G_\gamma}$ evaluated over
         $[0,T]\times M_1\times M_2$.
         In the periodic case, the index of $Q_{G_\gamma}$ restricted to the lifted profile fields is $\Ind_{\bar g}(\gamma)$.
       \end{lemma}
       \begin{proof}
           The lift is normal because $\bar g$ is conformal to the standard round metric, 
           while the $\cC$-total reality guarantees
           $df_i(TM_i)\perp_{\RR}\operatorname{span}_{\RR}\{f_i,Jf_i\}$.
           Moreover,
           \[
             \Vol(G_{\gamma_\varepsilon})
                =\Vol(M_1)\Vol(M_2)L_{\bar g}(\gamma_\varepsilon).
           \]
           Taking second variations gives the first equality in
           \eqref{eq62}
           and the second follows from
           $I_E=c_\gamma I_L$ for constant-speed geodesic.
       \end{proof}

          With the notation of \cref{se4},
           fix $T>0$.
         On $[0,T]$ we impose periodic boundary conditions
          when $T$ is a period of the profile,
          and Dirichlet boundary conditions otherwise.
          The same condition is imposed on all components of a profile variation.
          For a fixed-momentum variation $s_\varepsilon=s+\varepsilon\nu$,
          the Routh action has Hessian
                    \[
             Q_R(\nu)=\int_0^{T}
                \bigl(W_0\dot\nu^{\,2}-\mathcal V_R\nu^2\bigr)\,\dd t,
             \qquad
             \mathcal V_R=U_{\bmu}''+W_0'\ddot s+\frac12W_0''\dot s^{\,2},
          \]
          on either the periodic or Dirichlet domain.
            The associated Sturm--Liouville operator is
          \[
             L_R\nu=-\frac{\dd}{\dd t}(W_0\dot\nu)-\mathcal V_R\nu.
          \]
          For a full profile variation,
          take
          $V=\nu\partial_s+v_1\partial_{\theta_1}+v_2\partial_{\theta_2}$.

       \begin{lemma}[Square completion]
           \label{th7}
           Along an affinely parametrized $\bar g$-geodesic with momenta $\bmu$,
           \[
             I_E(V,V)
              =Q_R(\nu)+\sum_{i=1}^2\int_0^{T}
                W_i\left(\dot v_i+\frac{W_i'}{W_i}\dot\theta_i\nu\right)^2\dd t.
           \]
       \end{lemma}
       \begin{proof}
     Direct differentiation of the energy action gives
           $$
           I_E(V,V)=\int_0^T\Bigg[
              W_0\dot\nu^{2}
              +2W_0'\dot s\,\nu\dot\nu
              +\frac12W_0''\dot s^{2}\nu^2 \\
              +\sum_{i=1}^2\left(
                 W_i\dot v_i^{2}
                 +2W_i'\dot\theta_i\nu\dot v_i
                 +\frac12W_i''\dot\theta_i^{2}\nu^2
              \right)
           \Bigg]\,\dd t.
            $$
           The identities
           \[
           \begin{aligned}
           2W_0'\dot s\,\nu\dot\nu
              &=\frac{\dd}{\dd t}(W_0'\dot s\,\nu^2)
                -(W_0''\dot s^{2}+W_0'\ddot s)\nu^2,\\
           W_i\dot v_i^{2}+2W_i'\dot\theta_i\nu\dot v_i
              &=W_i\left(\dot v_i+
                 \frac{W_i'}{W_i}\dot\theta_i\nu\right)^2
                -\frac{(W_i')^2}{W_i}\dot\theta_i^{2}\nu^2
           \end{aligned}
           \]
           and $\mu_i=W_i\dot\theta_i$ imply
          $
              -U_{\bmu}''
              =\sum_{i=1}^2
              \left(\frac12W_i''-\frac{(W_i')^2}{W_i}\right)
              \dot\theta_i^{2}.
          $
           The total derivative has zero integral,
           by periodicity in the periodic case
           and because $\nu=0$ at both endpoints in the Dirichlet case.
           Substitution yields the formula.
       \end{proof}

        The index-iteration mechanism goes back to Bott
       \cite{Bott1956}; 
       see also \cite{Long2002}. 
        In the present
       setting, Routh reduction turns it into the following scalar Sturm
       count.

       \begin{proposition}[Geodesic-index growth]
           \label{coprofindex}
           Let $\gamma$ be an ordinarily closed regular oscillatory profile with $\bmu\ne0$. 
           Then
        $
            \Ind_{\bar g}(\gamma)\geq 2m_\gamma-3.
        $
           If exactly one momentum is nonzero, 
           then
           $\Ind_{\bar g}(\gamma)\geq  2m_\gamma-2$.
       \end{proposition}
       \begin{proof}
           Let $T=m_\gamma T_{\mathrm{cell}}$ 
           and 
           $E_R^-$ the  negative spectral space of $L_R$ with periodic boundary conditions.
           Since $L_R\dot s=0$ and $\dot s$ has $2m_\gamma$ simple zeros,
           periodic Sturm oscillation gives
           \[
             \dim E_R^-=\Ind(Q_R)\geq 2m_\gamma-1.
           \]
           Define two closing functionals
           $
             \mathfrak L_i(\nu)
             =\int_0^T\frac{W_i'}{W_i}\dot\theta_i\nu\,\dd t,
             \, i=1,2
           $
           and
           $K=E_R^-\cap\ker\mathfrak L_1\cap\ker\mathfrak L_2$. 
           So    $\dim K\geq 2m_\gamma-3$.
          In the singly spiral case,
          as we only need to require one $\mathfrak L_i$ to be zero (as the integrand for the other automatically vanishes),
          it follows that $\dim K\geq 2m_\gamma-2$.
           For $\nu\in K$, 
           we can  reconstruct periodic functions $v_i$ satisfying $\dot v_i=-\frac{W_i'}{W_i}\dot\theta_i\nu$ with $v_i(0)=0$
           and \cref{th7} gives $I_E(V,V)=Q_R(\nu)<0$. 
           Finally, normal projection preserves both dimension and negativity 
           because we get
           $I_E(V^\perp,V^\perp)= I_E(V,V) - c_\gamma^2\int_0^T|\dot \eta |^2\,\dd t $
           where   $V=V^\perp+\eta \dot\gamma$.
         Thus the statement follows.
       \end{proof}

      \subsection{Factor--profile splitting}
\label{se12}
   Let $t\in[0,T]$ parametrize the closed profile,
         and set $a_1=a$ and $a_2=b$.
         For $\xi_i\in\Gamma(\cN_{f_i}^{\Sphere})$,
             the lifts are then
             $    \widehat\xi_1=(z_1\xi_1,0),\,    \widehat\xi_2=(0,z_2\xi_2) $.
             The field $\widehat\xi_i$ is orthogonal to both factor tangent blocks,
             but its $Jf_i$-component overlaps the profile tangent $E_0$. 
             Its normal projection is therefore
$
    \mathcal F_i^\perp\xi_i
    :=
    \widehat\xi_i
    -
    \langle\widehat\xi_i,E_0\rangle E_0
    $
    and
$    \langle\widehat\xi_i,E_0\rangle
    =
    \frac{a_i^2\dot\theta_i}{v_\gamma}
    \langle\xi_i,Jf_i\rangle .
$
Set
    $    \mathscr C_i\xi_i  :=    \frac{d}{d\varepsilon}\big|_{0}
    f_{i,\varepsilon}^*\alpha
    $.
Then, for $X_i\in TM_i$, we have
\[
    \left\langle
    \partial_tG_{i,\varepsilon},
    dG_{i,\varepsilon}(X_i)
    \right\rangle
    =
    a_i^2\dot\theta_i
    \bigl(f_{i,\varepsilon}^*\alpha\bigr)(X_i),
    \quad
    \left.
    \frac{d}{d\varepsilon}
    \right|_0
    \left\langle
    \partial_tG_{i,\varepsilon},
    dG_{i,\varepsilon}(X_i)
    \right\rangle
    =
    a_i^2\dot\theta_i
    (\mathscr C_i\xi_i)(X_i).
\]

\begin{lemma}[Factor-Hessian splitting]
\label{pr9}
                   Assume $f_1,f_2$ are compact minimal $\cC$-totally real immersions 
                   and
                    $\gamma$ a closed profile for which $G_\gamma$ is a minimal immersion. 
                    For $\{i,j\}=\{1,2\}$, 
                    set
$
    \mathcal A_i
    =
    \Vol(M_j)L_{\bar g}(\gamma)>0,
    \,
    \mathcal B_i
    =
    \Vol(M_j)
    \int_0^T
    \frac{\rho a_i^2\dot\theta_i^2}{v_\gamma}\,\dd t
    \geq 0.
$
Then, for $i=1$ and $2$,
\begin{equation}\label{eq63}
    Q_{G_\gamma}\bigl(\mathcal F_i^\perp\xi_i\bigr)
    =
    \mathcal A_iQ_{f_i}(\xi_i)
    -
    \mathcal B_i
    \|\mathscr C_i\xi_i\|_{L^2(M_i)}^2.
\end{equation}
Moreover,
the two factor sectors and the lifted profile sector are pairwise
$Q_{G_\gamma}$-orthogonal.
\end{lemma}

\begin{proof}
Let $G_{i,\varepsilon}$ be the variation obtained by replacing $f_i$    with $f_{i,\varepsilon}$
and
$
    \widehat\xi_i =    \partial_\varepsilon G_{i,\varepsilon} \big|_0 .
$
Along a closed minimal immersion,
the volume Hessian depends only on the normal component of a variational field, namely
$
    Q_{G_\gamma}\bigl(\mathcal F_i^\perp\xi_i\bigr)
    =
    \frac{d^2}{d\varepsilon^2}
    \big |_0
    \Vol(G_{i,\varepsilon}).
$

Introduce
$
    h_{i,\varepsilon}
    =
    f_{i,\varepsilon}^*g_{\Sphere}, \,
    \beta_{i,\varepsilon}
    =
    f_{i,\varepsilon}^*\alpha.
$
The two factor blocks remain orthogonal.
Since $f_j^*\alpha=0$,
the only potentially nonzero mixed block involving the profile direction is
$
    \left\langle
    \partial_tG_{i,\varepsilon},
    dG_{i,\varepsilon}(X_i)
    \right\rangle
    =
    a_i^2\dot\theta_i
    \beta_{i,\varepsilon}(X_i).
$
With respect to the splitting
$
    \RR\partial_t\oplus TM_1\oplus TM_2,
$
the induced metric has the block form
\[
    g_{i,\varepsilon}
    =
    \begin{pmatrix}
        v_\gamma^2
        &
        a_i^2\dot\theta_i\beta_{i,\varepsilon}
        &
        0
        \\
        a_i^2\dot\theta_i\beta_{i,\varepsilon}^{\mathsf T}
        &
        a_i^2h_{i,\varepsilon}
        &
        0
        \\
        0
        &
        0
        &
        a_j^2g_j
    \end{pmatrix}.
\]
Thus we arrive at
$
    \dd\vol_{G_{i,\varepsilon}}
    =
    \rho
    \sqrt{
        v_\gamma^2
        -
        a_i^2\dot\theta_i^2
        |\beta_{i,\varepsilon}|_{h_{i,\varepsilon}}^2
    }\,
    \dd t\,
    \dd\vol_{h_{i,\varepsilon}}
    \dd\vol_{g_j}.
$
As $f_i^*\alpha=0$,
we get
$
    \beta_{i,0}=0
$,
$
    \partial_\varepsilon\beta_{i,\varepsilon}
    \big|_0
    =
    \mathscr C_i\xi_i,
$
and
$
    \partial_\varepsilon^2
    \big |_0
    \sqrt{
        v_\gamma^2
        -
        a_i^2\dot\theta_i^2
        |\beta_{i,\varepsilon}|_{h_{i,\varepsilon}}^2
    }
    =
    -
    \frac{a_i^2\dot\theta_i^2}{v_\gamma}
    |\mathscr C_i\xi_i|_{g_i}^2.
$
As a result,
\[
    Q_{G_\gamma}\bigl(\mathcal F_i^\perp\xi_i\bigr)
    =
    \Vol(M_j)
    \left(
        \int_0^T\rho v_\gamma\,\dd t
    \right)
    Q_{f_i}(\xi_i)
    -
    \Vol(M_j)
    \left(
        \int_0^T
        \frac{\rho a_i^2\dot\theta_i^2}{v_\gamma}\,\dd t
    \right)
    \|\mathscr C_i\xi_i\|_{L^2(M_i)}^2,
\]
which is \eqref{eq63}.

For mixed variations,
the factor-volume terms contain a first variation of
$\Vol(f_i)$
and hence vanish by minimality.
The two contact defects occur in separate squares.
This proves the asserted $Q_{G_\gamma}$-orthogonality.
\end{proof}

       \begin{theorem}[Additive index bound]
           \label{th10}
          Under the hypotheses of \cref{pr9},
           \begin{equation}\label{eq65}
             \Ind(G_\gamma)\geq 
             \Ind(f_1)+\Ind(f_2)+\Ind_{\bar g}(\gamma).
           \end{equation}
           In particular, for an ordinarily closed regular oscillatory profile
           with $\bmu\ne0$,
           \[
             \Ind(G_\gamma)\geq 
             \Ind(f_1)+\Ind(f_2)+2m_\gamma-3,
           \]
           with $2m_\gamma-2$ in the singly spiral case.
       \end{theorem}
       \begin{proof}
       We can lift negative spaces for $Q_{f_1}$ and $Q_{f_2}$, 
       together   with that of $\gamma$ of dimension $\Ind_{\bar g}(\gamma)$.  
           By \cref{pr8,pr9} these lifts form a $Q_{G_\gamma}$-orthogonal negative space. 
           The second inequality then follows from  \cref{coprofindex}.
       \end{proof}

   \subsection{Spherical quotient index}
          \label{se10}
           Suppose the hypotheses of \cref{th6} hold,
          and denote the descended primitive spherical quotient by
          $$
           \overline G_\gamma:
             \cM_G\longrightarrow\Sphere^{2n_1+2n_2+3}.
          $$
          Thus
          $
          G_\gamma(t+T_G,x,y)=G_\gamma(t,\Phi(x,y))
          $
          and
          $
          T_G=m_GT_{\mathrm{cell}}.
          $
          For a regular oscillatory profile,
          choose $t=0$ at a turning point and
         set $Q_R^D=Q_R|_{H_0^1(0,T_G)}$.

       \begin{lemma}[Seam approximation]
           \label{lm11}
           Every $H^1$ normal field on
           $[0,T_G]\times M_1\times M_2$ vanishing at $t=0,T_G$ 
           %continuously 
           closes up at the boundary and hence defines an $H^1$ normal field on the quotient.
           Every finite-dimensional negative space of such fields can be approximated,
           without changing dimension or negativity, by smooth quotient fields supported away from the seam.
       \end{lemma}
            \begin{proof}
           Zero endpoint traces satisfy every twisted seam condition.
           Smooth sections supported away from the two endpoint faces are dense in the zero-trace $H^1$ space.
            Approximating a basis of any finite-dimensional negative space and using continuity of the Jacobi form preserves both dimension and negative definiteness for sufficiently close approximants.
            These approximants vanish near the endpoints and hence descend to the quotient.
       \end{proof}

       \begin{proposition}[Spherical quotient index bound]
           \label{th9}
           Given compact minimal $\cC$-totally real immersion factors 
           and
           regular oscillatory profile $\gamma$ with $\bmu\ne0$,
           % the quotient $\overline G$ above,
           we get
           \[
             \Ind(Q_R^D)=2m_G-1,
             \qquad
             \Ind(\overline G_\gamma)\geq \max\{0,2m_G-3\}.
           \]
           For the singly spiral case the last term improves to
           $\max\{0,2m_G-2\}$.
       \end{proposition}
       \begin{proof}
    %  Assume that $t=0$ is a turning point by translating $t$.
Since
$ L_R\dot s=0,  \,    \dot s\in H_0^1(0,T_G)$,
and $\dot s$ has $2m_G-1$ simple interior zeros,
Sturm oscillation gives $\Ind(Q_R^D)=2m_G-1$.
            As in the proof of \cref{coprofindex},
           we can reconstruct $v_i$ by
           $ \dot v_i =    -\frac{W_i'}{W_i}\dot\theta_i\nu$,
           with   $v_i(0)=v_i(T_G)=0$
       losing  at most two dimensions in general, 
       and at most one in the singly spiral case.
                  They therefore provide negative spaces of dimensions at least $2m_G-3$ and $2m_G-2$, respectively.
                      According to \cref{pr8,lm11}, 
                      their normal projections transfer to negative spaces of the same dimensions on the quotient.
       \end{proof}

%\majorsectiongap
\section{\texorpdfstring{Singly Spiral and Contact Levels}{Singly Spiral and Contact Levels}}
          \label{se14}

        This section focuses on when closed spiral minimal products have embedded images.
        Ordinary closing is insufficient,
        since additional coincidences may arise from different profile branches
        or from phase incidence of the factors.
       We use the term \emph{real inputs}
       for compact connected embedded minimal submanifolds
       in $\Sphere^{n_i}\subset \RR^{n_i+1}$,
       automatically viewed as $\cC$-totally real
       in $\Sphere^{2n_i+1}\subset\CC^{n_i+1}$.
        
         At $C_1=0$,
         we characterize exactly when a singly spiral product factors through a compact embedding:
              the rotating factor must be phase-saturated,
              and the profile period must satisfy the corresponding
              arithmetic condition determined by the order
              of its phase stabilizer.
              For real inputs,
              this order is one or two.
         For general phase-saturated inputs,
         every open set in the admissible parameter region
         that meets the line $C_1=0$
         contains infinitely many embedded doubly spiral minimal products.

               At $C_1=-1$, 
                every open interval above the contact threshold contains infinitely many embedded products for real inputs,
                while arbitrary prescribed embedded factors admit analogous families near the threshold.
                     In maximal contact dimension, 
                     a constant phase rotation makes these products special Legendrian.
                     Their return orders, Morse indices, and volumes are unbounded, 
                     while their second fundamental forms remain uniformly bounded.

       %We orient nonzero-momentum profiles by $\mu_1>0$.

   \subsection{\texorpdfstring{The singly spiral level $C_1=0$}
                              {The singly spiral level C1=0}}
          \label{se15}

        Now $\mu_2=0$.
          For $k_2>0$,
          the profile oscillates between two turning points in
          $0<s<\frac{\pi}{2}$;
          for $k_2=0$, the signed coordinate of \cref{rm1} gives a smooth
          passage through $u=0$, whose polar description inserts the phase
          shift $\pi$ described in \cref{lm6}.

          If $k_2=0$ and $k_1>0$,
           the two limits in \eqref{eq48}--\eqref{eq49} show that
  $\frac{\Delta_1(0,C_2)}{2\pi}$ covers an interval.
           Thus
           its rational values yield infinitely many
           closed profiles passing through a coordinate circle
           of unbounded primitive ordinary return order.

          For the next result,
           every regular oscillatory profile is understood on
           its maximal continuation $\gamma:\RR\to\Sphere^3$,
            and $\widetilde M_\gamma:=\RR\times M_1\times M_2$ denotes the full
           parameter domain of the spiral product.
          We say that $G_\gamma:\widetilde M_\gamma\to
          \Sphere^{n}$
           \emph{factors through a compact embedding}
           if $G_\gamma=\overline G\circ\mathfrak q$
            for a compact manifold $M_\gamma$ of dimension $\dim\widetilde M_\gamma$,
            a smooth covering map
            $\mathfrak q:\widetilde M_\gamma\to M_\gamma$,
            and an embedding $\overline G:M_\gamma\hookrightarrow
            \Sphere^{n}$.

               Recall from \cref{df5} that its phase stabilizer is
    $
        \mathsf H_f
        =\{\zeta\in\Sphere^1:\zeta f(M)=f(M)\}.
    $
       \begin{definition}[Incident phases and phase saturation]
    \label{dfps}
    Let $f:M\hookrightarrow\Sphere^{2n+1}$ be an embedding.
    Its incident phase set is
   $
        \mathcal P_f
        :=
        \{\zeta\in\Sphere^1:
        \zeta f(M)\cap f(M)\neq\varnothing\}.
   $
    We call $f$ \emph{phase-saturated} if
    $  \mathcal P_f=\mathsf H_f$.
\end{definition}

   Minimal submanifolds of the standard round sphere are real analytic
           \cite{Morrey1966}.
          Hence,
           if a compact connected embedded minimal submanifold $L$
           has the same image germ as $\zeta L$ at one point,
           then the identity principle gives $\zeta L=L$.

       \begin{proposition}[Embedded singly spiral quotients]  
       Assume
           $
                              f_1:M_1^{k_1}\hookrightarrow
                              \Sphere^{2n_1+1}
           $
           is a compact connected embedded minimal $\cC$-totally real submanifold
           and
           $f_2: M_2^{k_2}\hookrightarrow
                              \Sphere^{n_2+1}$
           a compact embedded minimal submanifold with $k_2>0$.
           Consider a regular oscillatory profile on $C_1=0$,
            oriented so that $\Delta_1>0$,
            and denote
          $   d=|\mathsf H_{f_1}|,   \,     \zeta_d=e^{\frac{2\pi}{d}i}$.
           Then $G_\gamma$ factors through a compact embedding
            if and only if $f_1$ is phase-saturated
            and
      $
                             \Delta_1=\frac{2\pi}{ \ell d} $
                              for some $\ell\in\NN$.
           In this case
   $
                              m_G=\ell,
                              m_\gamma=d\ell,
   $
           and the primitive spherical quotient has domain
           \begin{equation}\label{eqmt}
                              \left(
                              \frac{[0,\ell T_{\mathrm{cell}}]\times M_1}
                              {(\ell T_{\mathrm{cell}},x)
                               \sim(0,\varphi_{\zeta_d}(x))}
                              \right)\times M_2.
           \end{equation}
           Its canonical $d$-fold cyclic cover is
            $S^1\times M_1\times M_2$.
       \end{proposition}

       \begin{proof}
           At a fixed nonturning magnitude,
           phase differences between same oriented branches lie in
           $\Delta_1\mathbb Z$,
           while those between oppositely oriented branches
           lie strictly between consecutive elements
           of $\Delta_1\mathbb Z$.
           Hence any phase in
           $\mathcal P_{f_1}\setminus\mathsf H_{f_1}$
           produces a self-intersection
           which cannot be removed by taking a quotient.
           Therefore,
           if $G_\gamma$ factors through a compact embedding,
           then $\mathcal P_{f_1}=\mathsf H_{f_1}$,
           i.e., $f_1$ is phase-saturated.
    Since $\mathsf H_{f_1}=\langle\zeta_d\rangle$,
    if $\frac{2\pi}{d}\notin\Delta_1\mathbb N$,
    the phase $\zeta_d$ produces an opposite-branch collision
    with distinct tangent directions,
    contradicting the embeddedness.
    Thus
    $\ell\Delta_1=\frac{2\pi}{d}$
    for some $\ell\in\mathbb N$.

           Conversely,
           assume that $f_1$ is phase-saturated with
           $\ell\Delta_1=\frac{2\pi}{d}$.
           Since
           $\mathcal P_{f_1}=\mathsf H_{f_1}=\langle\zeta_d\rangle$,
           every incident phase is then a multiple of
           $\frac{2\pi}{d}=\ell\Delta_1$.
           As a result,
           the first return occurs after $\ell$ complete cells,
           and
           $
           G_\gamma(t+\ell T_{\mathrm{cell}},x,y)
           =
           G_\gamma(t,\varphi_{\zeta_d}(x),y).
           $
           Thus $G_\gamma$ induces the stated embedded primitive spherical quotient
           with $m_G=\ell$ and $m_\gamma=d\ell$.
          \end{proof}

%\noindent\emph{Application to real inputs.}
   \begin{corollary}[Real-input trichotomy]
\label{coreal}
       Assume that the embedding factor $f_1$ is real with image $L_1$.
        Since
       $\mathcal P_{f_1}\subset\{1,-1\}$,
        exactly one of the following occurs.
       \begin{enumerate}[label=\textup{(\roman*)}]
           \item
           If $L_1=-L_1$,
            then
       $
                              \mathcal P_{f_1}
                               =\mathsf H_{f_1}
                               =\{1,-1\},
                        \, d=2.
       $
           Hence $G_\gamma$ factors through a compact embedding
            exactly when
         $
                              \ell\Delta_1=\pi
                             \, \text{for some }\ell\in\NN.
        $
           In this case
          $
                              m_G=\ell, \,
                              m_\gamma=2\ell.
          $
          % and the primitive quotient is the mapping torus of
          % $\varphi_{-1}$ with canonical double cover
          % $\Sphere^1\times M_1\times M_2$.

           \item
           If $L_1\cap(-L_1)=\varnothing$,
            then
       $
                              \mathcal P_{f_1}
                               =\mathsf H_{f_1}
                               =\{1\},
                               \,d=1.
     $
           Hence $G_\gamma$ factors through a compact embedding
            exactly when
     $
                              \ell\Delta_1=2\pi
                              \, \text{for some }\ell\in\NN
      $
          with
       $
                              m_G=m_\gamma=\ell.
        $
     %      and the primitive quotient has domain
    %       $\Sphere^1\times M_1\times M_2$.

           \item
           If
         $
                              L_1\cap(-L_1)\ne\varnothing,
                            \,
                              L_1\ne-L_1,
         $
           then
         $
                              \mathcal P_{f_1}=\{1,-1\},
                \,
                              \mathsf H_{f_1}=\{1\}.
        $
           Thus $f_1$ is not phase-saturated,
            and no such singly spiral product factors through
            a compact embedding.
                   \end{enumerate}
       \end{corollary}
       
       \begin{remark}[Lawson surfaces   \cite{Lawson1970}]\label{rmLaw}
           For Lawson's embedded surface
           $\xi_{m,k}\subset\Sphere^3$,
       we have
           $
                              -\xi_{m,k}=\xi_{m,k}
            $
            if and only if
            $
                              m\equiv k\pmod 2.
           $
           Thus,
            if $m$ and $k$ are both even or both odd,
            then $\xi_{m,k}$ is antipodally invariant
            and belongs to case~\textup{(i)}.
           If one of $m,k$ is even and the other is odd,
                    $
                              -\xi_{m,k}\ne\xi_{m,k}.
         $
          Now
           Frankel's theorem   \cite{Frankel1966} therefore asserts
         $
                              \xi_{m,k}\cap(-\xi_{m,k})\ne\varnothing.
        $
           Hence these surfaces belong to the mixed
            antipodal case \textup{(iii)}.         
       \end{remark}

         \begin{remark}[Spherical isoparametric examples]\label{rmIso}
             Let $g$ be the number of distinct principal curvatures.
             By M\"unzner's theorem,
             the family admits a normalized homogeneous Cartan--M\"unzner polynomial $F$ of degree $g$,
                 and contains a unique minimal regular leaf and two minimal focal manifolds \cite{Munzner1980,Munzner1981}.       
                 Since $F(-x)=(-1)^gF(x)$,
every level is antipodally invariant when $g$ is even.
When $g$ is odd,
the antipodal map exchanges  two focal manifolds.
Hence the minimal regular leaf is always in case \textup{(i)},
while the focal manifolds are in case~\textup{(i)}
for even $g$ and in case \textup{(ii)} for odd $g$.
\end{remark}

         Next result is not exactly limited to the slice $C_1=0$ 
         but instead
         tells us that there are abundant examples
         of embedded spiral minimal products for closed $\gamma$
         arising from parameters near $\{C_1=0\}$
         in the admissible region.
          
       \begin{theorem}[Embedded closing near       $C_1=0$]
           \label{co8}
           Assume $k_2>0$
            and
             $f_i:M_i^{k_i}\hookrightarrow\Sphere^{2n_i+1}$
            be compact,
             connected,
             embedded,
             minimal,
            phase-saturated $\cC$-totally real submanifolds.
        Every nonempty open set
$\mathcal O\subset\mathcal U_{k_1,k_2}$
meeting the singly spiral line $C_1=0$
contains a sequence of ordinarily closed doubly spiral parameters
      $\lambda_j=(C_{1,j},C_{2,j})$,
            with $C_{1,j}>0$,
             whose primitive spherical quotients are compact   embedded minimal submanifolds
       $
                              \overline G_j:\cM_j
                              \hookrightarrow\Sphere^{2n_1+2n_2+3}.
        $
       \end{theorem}

       \begin{proof}
        Fix $\mu_1>0$.
                     Since $C_2\mapsto\Delta_1(0,C_2)$ is analytic and nonconstant,
           choose $\lambda_*=(0,C_{2,*})\in\mathcal O$ with $\partial_{C_2}\Delta_1(\lambda_*)\ne0$.
           By \eqref{eq50}--\eqref{eq51},
           $\Pi$ is a local diffeomorphism around $\lambda_*$.
          Let $\Pi(\lambda_*)=(\rho_*,0)$
                 and 
               $d_2=|\mathsf H_{f_2}|$.
           Choose integers $m_j\to\infty$ divisible by $d_2$
           and integers $a_j$ such that
           $\frac{a_j}{m_j}\to\rho_*$.
           After restricting the inverse neighborhood to $\mathcal O$,
           define
         %  for all large $j$,
         $
                              \lambda_j  :=\Pi^{-1}\left(
                                                   \frac{a_j}{m_j},\frac1{m_j}\right)
          $
for large $j$.

          Since $\Delta_2$ has the sign of $C_1$ by \eqref{eq47},
           $C_{1,j}>0$.
           The second phase has exact order $m_j$,
           so $m_{\gamma,j}=m_j$.
           For  fixed $j$,
           we omit  the index $j$
           %simply 
           and  write
           $
           (C_1,C_2)=\lambda_j, \,
           m=m_j, \,
           \Delta_i=\Delta_i(\lambda_j),
           \,
           u=(e^{i\Delta_1},e^{i\Delta_2}).
           $
           For $s\in(s_-,s_+)$,
           \[
                              \beta(s)
                              =2C_1\int_{s_-}^s
                              \frac{\cot v}
                              {\sqrt{D_{C_1,C_2}(v)}}\,\dd v.
           \]
Since $d_2\mid m$,
we have $\mathsf H_{f_2}
\subset
\left\langle e^{\frac{2\pi i}{m}}\right\rangle$.
If two opposite branches met,
then there exists some $r\in\mathbb Z$ such that
$
e^{i\beta(s)}
\in
e^{ir\Delta_2}\mathsf H_{f_2}
\subset
\left\langle e^{\frac{2\pi i}{m}}\right\rangle .
$
This is impossible because
$\Delta_1=\frac{2a_j\pi}{m}$,
$\Delta_2=\frac{2\pi}{m}$ by the definition of $\Pi$
and
$0<\beta(s)<\frac{2\pi}{m}$.
Thus opposite interior branches are disjoint.
By phase saturation,
every same-oriented coincidence is a spherical return.
Hence the primitive spherical quotient is embedded.
       \end{proof}

\begin{example}[Antipodally invariant real inputs]
    Let  two factors in \cref{co8} be real and antipodally invariant.
    Such inputs include the same-parity Lawson surfaces in
    \cref{rmLaw}
    and the minimal regular leaves of spherical isoparametric families
    in \cref{rmIso}.
    Then
    $
        \mathsf H_{f_1}
        =
        \mathsf H_{f_2}
        =
        \{1,-1\}.
    $
    Denote the induced antipodal involution on $M_i$
    by $\tau_i$.

    In the rational closing used in the proof of \cref{co8},
    the denominator $m_j$ is even,
    while the numerator $a_j$ may be chosen even or odd
    without affecting the approximation.
    After half the ordinary period,
    the phase return is $(1,-1)$ when $a_j$ is even
    and $(-1,-1)$ when $a_j$ is odd.
    Thus every $\mathcal O$ as in \cref{co8}
    contains infinitely many embedded spherical quotients
    with gluing map
    $
        \mathrm{id}_{M_1}\times\tau_2,
    $
    and infinitely many with gluing map
    $
        \tau_1\times\tau_2.
    $

    For the real great spheres
    $M_i=\Sphere^{k_i}$,
    with $k_1,k_2>0$ and $k_1$ even,
    exactly one of these two spherical quotients is orientable.
    Hence they are not diffeomorphic.
\end{example}

The separate phase-saturation hypothesis can be weakened to a joint
condition along the chosen profile;
see \cref{rmmon,expp}.

   \subsection{\texorpdfstring{The contact level $C_1=-1$}
                              {The contact level C1=-1}}
          \label{se16}

       \begin{proposition}[Contact criterion]
           \label{pr12}
           For  $\cC$-totally real inputs,
            the spiral product is $\cC$-totally real
            exactly when 
            $  \alpha_\gamma(\dot\gamma)=0$,
                  i.e., $
                              \mu_1+\mu_2=0.
           $
           On the chart $\mu_1\ne0$,
            this is $C_1=-1$.
           If both connected minimal inputs have maximal contact dimension,
            then, up to  a  phase rotation,  the output
          is  special Legendrian.
       \end{proposition}
       \begin{proof}
           The contact form vanishes on the factor directions,
            and on the profile direction it is
            $\frac{\mu_1+\mu_2}{\rho^2}$.
       \end{proof}

          In what follows,
          a contact profile means a profile on $C_1=-1$,
          oriented so that $\mu_1>0$.
          Then $\mu_2=-\mu_1<0$,
          so it avoids both coordinate circles;
          if nonsteady,
          it is regular oscillatory by \cref{pr3}.

   \subsubsection{The phase character and noncrossing}
          \label{se17}

With $p=k_1+1$ and $q=k_2+1$,
introduce the phase character
\begin{equation}\label{chipq}
\chi_{p,q}:\TT^2\longrightarrow \Sphere^1,\qquad
\chi_{p,q}(\zeta_1,\zeta_2)=\zeta_1^p\zeta_2^q.
\end{equation}
It is the phase combination singled out by the contact reduction and will distinguish
same-branch returns from opposite-branch incidences.
      
        \begin{lemma}[Exact phase-character law]
           \label{lm12}
           Let $\gamma$ be a nonsteady contact profile
           with reduced energy $C_2$.
           Choose phase lifts
           $\theta_1,\theta_2:\RR\to\RR$
           and set
           \[
                              \Psi(t):=
                              p\theta_1(t)+q\theta_2(t),
                               \qquad
                              D_{C_2}(s):=
                              C_2\cos^{2p}s\sin^{2q}s-1,
                               \qquad
                              \alpha(s):=
                              \arctan\sqrt{D_{C_2}(s)}.
           \]
           Here $s\in[s_-,s_+]$.
           Then there is a constant $\Psi_0$ such that,
           away from the turning points,
           \[
                              \Psi(t)
                              =
                              \begin{cases}
                              \Psi_0-\alpha(s(t)),
                               &\dot s(t)>0,\\[1mm]
                              \Psi_0+\alpha(s(t)),
                               &\dot s(t)<0.
                              \end{cases}
           \]
           These formulas extend continuously to the turning points,
           where $\Psi=\Psi_0$.
           Consequently,
           \begin{equation}\label{eq74}
                              p\Delta_1+q\Delta_2
                              =
                              0.
           \end{equation}
           If $s(t)=s(t')=s\in(s_-,s_+)$
           and $\dot s(t)>0>\dot s(t')$,
           then
           \[
                              \chi_{p,q}
                              \left(
                              e^{i(\theta_1(t')-\theta_1(t))},
                              e^{i(\theta_2(t')-\theta_2(t))}
                              \right)
                              =
                              e^{2i\alpha(s)}
                              \notin\{\pm1\}.
           \]
           More generally,
           the phase difference between any two oppositely oriented branches
           of magnitude $s$
           has character $e^{\pm2i\alpha(s)}\notin\{\pm1\}$.
       \end{lemma}

       \begin{proof}
           On every increasing half-cell,
           \eqref{eq25} with $C_1=-1$ gives
           \[
                              \frac{\dd\Psi}{\dd s}
                              =
                              \frac{p\tan s-q\cot s}{\sqrt{D_{C_2}(s)}}
                              =
                              -\frac{\dd\alpha}{\dd s}.
           \]
    %       where the last equality follows by differentiating
   %        $D_{C_2}(s)+1=C_2\cos^{2p}s\sin^{2q}s$.
           On every decreasing half-cell the sign is reversed.
           Since $\alpha(s_-)=\alpha(s_+)=0$,
           continuity gives the same integration constant at every turn.
           This proves the two branch formulas and \eqref{eq74}.
           Finally,
           subtracting the two branch formulas gives
           $\Psi(t')-\Psi(t)=2\alpha(s)$.
           Complete-cell shifts do not change this character by \eqref{eq74},
           which proves the general assertion.
           As $0<\alpha(s)<\frac{\pi}{2}$
           for $s_-<s<s_+$,
           the resulting character value is neither $1$ nor $-1$.
       \end{proof}

\begin{lemma}[Calibrated-plane rigidity]\label{lmcal}
           For a smooth calibration form $\phi$ of degree $d$, 
           any two  $\phi$-calibrated $d$-planes at a point
          sharing the same oriented $(d-1)$-plane must coincide.
       \end{lemma}

       \begin{proof}
           If $\xi$ orients the common hyperplane
           and $v_1,v_2$ are its positive unit complements,
           then
           $1=\phi\left(\xi\wedge\frac{v_1+v_2}{2}\right)
           \leq  \left|\frac{v_1+v_2}{2}\right|\leq 1$.
           Hence $v_1=v_2$.
       \end{proof}

       \begin{proposition}[Noncrossing of contact profiles]
    \label{th13}
    Let $\gamma:\RR\to\Sphere^3$ be a nonsteady contact profile,
    affinely parametrized for $\bar g$.
    If $t_1\ne t_2$ and $\gamma(t_1)=\gamma(t_2)$, then
    \[
        \gamma(t+t_2-t_1)=\gamma(t)
        \qquad\text{for every }t\in\RR.
    \]
    Consequently, $\gamma$ is either (i) injective or (ii) periodic with embedded image.
\end{proposition}

\begin{proof}
    Couple $\gamma$ to any pair of connected compact special Legendrian submanifolds, e.g., the standard real Legendrian factors
    $\Sphere^{k_i}\subset\Sphere^{2k_i+1}$
    (taking one point if $k_i=0$).
    The resulting product is minimal Legendrian,
    so a constant phase rotation makes its cone $\mathcal C$
    special Lagrangian with a fixed calibration.
    At a self-incidence of $\gamma$,
    the two calibrated tangent planes obtained by fixing the factor points
    share the oriented subspace spanned by the radial and factor directions.
    By \cref{lmcal},
    their remaining oriented directions agree.
    Hence $\dot\gamma(t_1)=\dot\gamma(t_2)$,
    and weighted-geodesic uniqueness gives the  periodicity.
\end{proof}

   \subsubsection{Contact phase limits}
          \label{se18}

       \begin{lemma}[Lowest-energy contact limit]
           \label{lm15}
           Let $p,q\geq 1$.
           The unique minimizer $s_*$ of $P_{-1}$ satisfies
           $\cos^2s_*=\frac{p}{p+q}$.
           Set $    P_*:=P_{-1}(s_*)
                              =\frac{(p+q)^{p+q}}{p^p q^q}$.
           The nonsteady contact levels are exactly $C_2>P_*$;
            as $C_2\downarrow P_*$,
             their two turning points coalesce at $s_*$,
            and
           \[
                              s_\pm\longrightarrow s_*,
                               \qquad
                              \Delta_1\longrightarrow\pi\sqrt{\frac{2q}{p(p+q)}},
                               \qquad
                              |\Delta_2|\longrightarrow\pi\sqrt{\frac{2p}{q(p+q)}}.
           \]
       \end{lemma}

       \begin{lemma}[High-energy contact limit]
           \label{lm16}
           Let $p,q\geq 1$.
           Along the contact levels $C_2>P_*$,
           \[
                              \Delta_1\longrightarrow\frac\pi p,
                               \qquad |\Delta_2|\longrightarrow\frac\pi q
                               \qquad(C_2\uparrow \infty).
           \]
       \end{lemma}

          Proofs of both lemmas appear in \cref{ap5}.
          These are the $C_1=-1$ counterparts of the $C_1=0$ endpoint
           limits \eqref{eq48}--\eqref{eq49}.

   %    \begin{remark}[Symmetric contact phase monotonicity]
 %          \label{rm4}
 %%          If $p=q\geq 2$,
  %          then $s_+=\frac{\pi}{2}-s_-$,
  %          $\Delta_2=-\Delta_1$, and $C_2\mapsto\Delta_1(C_2)$ is strictly
  %          decreasing from $\frac{\pi}{\sqrt p}$
  %          to $\frac{\pi}{p}$;
    %        see \cref{ap6}.
  %         For $p\ne q$,
 %           global monotonicity is not known in general.
 %  %        When $p=q=1$,
  %          the profile is simply an unweighted great circle 
  %         % under the connected point-factor convention,
   %        with  $(\Delta_1,\Delta_2)=(\pi,-\pi)$ and $m_\gamma=2$.
 %      \end{remark}

   \subsubsection{Odd-prime profiles and real-input products}
          \label{se19}

          For real inputs,
          every incident phase belongs to $\{1,-1\}$;
          hence every possible pair of factor phases lies in the
          coordinate-sign subgroup
          $\TT^2[2]:=\{\pm1\}^2$
          (acting on $\Sphere^3\subset\CC^2$ coordinatewise).
          The next two results concern the profile alone,
          but are tailored to excluding the nontrivial elements
          of this subgroup.

       \begin{lemma}[Coordinate-sign avoidance]
           \label{co5}
           
           An ordinarily closed nonsteady contact profile is disjoint from
%each of
     its images under the nonidentity elements of $\TT^2[2]$
if and only if $m_\gamma$ is odd.
       \end{lemma}
       \begin{proof}
           Set $u=(e^{i\Delta_1},e^{i\Delta_2})$.
           By \cref{th4},
           $\ord(u)=m_\gamma$.
           The last assertion of \cref{lm12} excludes an incidence between opposite branches with phase in $\TT^2[2]$.
           Hence every coordinate-sign incidence is an $r$-cell return,
           so $u^r\in\TT^2[2]$ for some $r\in\ZZ$.
           If $m_\gamma$ is odd,
           then $u^{2r}=1$ implies $m_\gamma\mid r$,
           and the sign is trivial.
           If $m_\gamma$ is even,
           then $u^{\frac{m_\gamma}{2}}$
           is a nontrivial element of $\TT^2[2]$,
           proving the converse.
       \end{proof}

  \begin{proposition}[Odd-prime profile returns]
    \label{th14}
    Let $p+q>2$
    and $\mathcal U\Subset(P_*,\infty)$
    be a nonempty open interval.
    For every sufficiently large odd prime $\ell$,
    there is an energy $C_{2,\ell}\in\mathcal U$
    whose corresponding contact profile $\gamma_\ell$
    is ordinarily closed with
    $m_{\gamma_\ell}=\ell$.
    Moreover,
    $\gamma_\ell(\RR)$ is an embedded circle and
$
        \gamma_\ell(\RR)\cap
        \varepsilon\gamma_\ell(\RR)=\varnothing
$
       for every
       $        \varepsilon\in\TT^2[2]\setminus\{(1,1)\}$.
\end{proposition}
       \begin{proof}
       By \cref{pr5,lm15,lm16},
       the first component $C_2\mapsto\frac{\Delta_1(-1,C_2)}{2\pi}$ of the restricted phase map $\Pi(-1,\cdot)$ is real analytic and nonconstant on $(P_*,\infty)$,
       and hence its image on $\mathcal U$ contains an open interval $J$.
           Set
           $(p_0, q_0)=\frac{(p, q)}{\gcd(p,q)}$.
           For every sufficiently large odd prime $\ell$,
           choose $\beta_\ell\in\ZZ$ with
           $\frac{q_0\beta_\ell}{\ell}\in J$
           and $\ell\nmid\beta_\ell$.
           We may also assume $\ell\nmid p_0q_0$.
           Select $C_{2,\ell}\in\mathcal U$ so that
                  $
                              \frac{\Delta_1}{2\pi}
                              =\frac{q_0\beta_\ell}{\ell}.
         $
           By \eqref{eq74},
           $\frac{\Delta_2}{2\pi}
           =-\frac{p_0\beta_\ell}{\ell}$.
           Both fractions have denominator $\ell$ in lowest terms,
           so \cref{th4} gives $m_\gamma=\ell$.
           The last assertions follow from \cref{th13,co5}.
       \end{proof}

          Coupling these profiles to real inputs gives the following application.

\begin{corollary}[Embedded products from real inputs]
    \label{th15}
    Let
    $f_i:M_i^{k_i}\hookrightarrow\Sphere^{n_i}$
    %$i=1,2$,
    be real inputs with $k_1+k_2>0$.
    Every nonempty open interval $\mathcal I\subset(P_*,\infty)$ contains a sequence of contact energies $C_{2,j}$
    for which the spiral products define  minimal $\cC$-totally real embeddings
       $  G_j:   S^1\times M_1\times M_2   \hookrightarrow    \Sphere^{2n_1+2n_2+3} $
    with
$
        m_{G,j}=m_{\gamma,j}=\ell_j
$
being odd primes
    and $\ell_j\nearrow \infty$.
\end{corollary}

       \begin{proof}
           Choose $\mathcal U\Subset\mathcal I$
and apply \cref{th14}.
           Every incident phase of a real factor belongs to $\{\pm1\}$,
           so \cref{co5} excludes every nontrivial product coincidence
           and premature spherical return.
          % Minimality and total reality follow from \cref{th1,pr12}.
       \end{proof}

   \subsubsection{Factor-adapted products for general inputs}
          \label{se20}

We now drop the reality assumption.
For arbitrary embedded $\cC$-totally real factors,
the stabilizers $\mathsf H_{f_i}$ determine spherical closing.
After closing,
embeddedness requires avoiding all additional phase incidences,
which belong to the finite set
$\mathcal P_{f_1}\times\mathcal P_{f_2}$.

       \begin{lemma}[Finiteness of phase incidence]
           \label{lm18}
           If $f:M\hookrightarrow\Sphere^{2n+1}$ is a compact
            minimal $\cC$-totally real embedding,
            then $\mathcal P_f$ is finite.
       \end{lemma}
       \begin{proof}
           See \cref{ap9}.
       \end{proof}

With $d_0:=\gcd(p,q)$ and
$(p_0, q_0)=\frac{(p,q)}{d_0}$,
 the character function \eqref{chipq} has kernel
\begin{equation}\label{kerchi}
\ker\chi_{p,q}
=
\bigsqcup_{j=0}^{d_0-1}
\left(e^{\frac{2\pi j} p i},1\right)K_0,
%\qquad
 \text{ where }
K_0:=(\ker\chi_{p,q})^\circ
=
\left\{
\left(e^{iq_0\tau},e^{-ip_0\tau}\right):
\tau\in\RR
\right\}
\end{equation}
For a contact profile,
let
$u(C_2):=(e^{i\Delta_1},e^{i\Delta_2})$
be its cell-phase increment.
Since \eqref{eq74} is an exact identity in $\RR$,
the path
$r\mapsto(e^{ir\Delta_1},e^{ir\Delta_2})$,
$0\leq  r\leq 1$,
lies in $\ker\chi_{p,q}$
and joins $(1,1)$ to $u(C_2)$.
Hence the cell-phase map takes values in $K_0$.
In contrast,
the opposite-branch phase differences lie outside
$\ker\chi_{p,q}$ by \cref{lm12}.

       \begin{theorem}[Factor-adapted embedded products]
           \label{th16}
           Let
         $
                              f_i:M_i^{k_i}\hookrightarrow\Sphere^{2n_i+1}
        $
           be compact connected embedded minimal $\cC$-totally real submanifolds
            with $k_1+k_2>0$.
           There is a compact interval
           $\mathcal K\Subset(P_*,\infty)$ such that,
           for every sufficiently large prime $\ell$,
           one can choose a contact energy $C_{2,\ell}\in\mathcal K$
           of ordinary return order $\ell$
           whose associated spiral product
           \[
                              G_\ell:
                              S^1\times M_1\times M_2
                              \hookrightarrow\Sphere^{2n_1+2n_2+3}
           \]
           is embedded
           minimal
           $\cC$-totally real.
           Moreover,
           $m_{G,\ell}=m_{\gamma,\ell}=\ell$.
           If $k_i=n_i$ for $i=1,2$,
           a constant phase rotation of $G_\ell$ is embedded special Legendrian.
       \end{theorem}
       \begin{proof}
           Set
           $
                              \chi=\chi_{p,q},\,
                              \mathcal P=\mathcal P_{f_1}\times\mathcal P_{f_2},
 \,                               K_0=(\ker\chi)^\circ.
           $
           By \cref{lm18},
           $\mathcal P$ is finite.
           Hence, for every sufficiently large prime $\ell$,
           the order-$\ell$ subgroup $K_0[\ell]$ satisfies
           $
                              K_0[\ell]\cap\mathcal P=\{(1,1)\}.
           $
           Finiteness also gives $d_*>0$ such that
           $
                              \dist_{\Sphere^1}(1,\chi(\zeta))\geq  d_*
                              $
                            for every 
                            $\zeta\in\mathcal P$
                              with 
                             $\chi(\zeta)\ne1;$
           if there is no such $\zeta$, take $d_*=\pi$.
           Since
           $D_{C_2}(s)\leq   \frac {C_2} {P_*}-1$,
           choose a nonempty interval
           $\mathcal U\Subset(P_*,\infty)$ sufficiently close to $P_*$ that
        $
                              2\arctan\sqrt{D_{C_2}(s)}<d_*
         $
           along every profile with $C_2\in\mathcal U$,
           and take $\mathcal K=\overline{\mathcal U}$.

           By \cref{th14},
           for every sufficiently large prime $\ell$
           there is $C_{2,\ell}\in\mathcal U$ with
           $m_{\gamma,\ell}=\ell$.
           If
           $u_\ell=(e^{i\Delta_1},e^{i\Delta_2})$,
           then \eqref{eq74} gives $u_\ell\in K_0$ and
           $\langle u_\ell\rangle=K_0[\ell]$.
           Same-oriented phase differences,
           as well as differences between distinct occurrences of a turning
           point,
           are powers of $u_\ell$.
           None belongs to $\mathcal P$ before the full return.
           For oppositely oriented interior branches,
           by \cref{lm12} and our choice above, it follows that
           \[
                              0<\dist_{\Sphere^1}(1,\chi(\zeta))
                              =2\arctan\sqrt{D_{C_{2,\ell}}(s)}<d_*,
           \]
           so $\zeta\notin\mathcal P$ by the definition of $d_*$.
           Since $\mathsf H_{f_1}\times\mathsf H_{f_2}\subset\mathcal P$,
           the preceding exclusions show simultaneously 
           that $G_\ell$ is injective and that no spherical return occurs
           before the full $\ell$-cell return.
           Hence
$
m_{G,\ell}=m_{\gamma,\ell}=\ell.
$
Since $G_\ell$ is an immersion of a compact manifold,
it is an embedding.
           The remaining assertions follow from \cref{th1,pr12}.
       \end{proof}

\begin{remark}[Quantitative consequences]
By \cref{th10},
$
\Ind(G_\ell)
\geq 
\Ind(f_1)+\Ind(f_2)+2\ell-3.
$
Since $\mathcal K\Subset(P_*,\infty)$,
\cref{co1} and \eqref{eq8} give constants $c_V,C_V,  C_B>0$,
independent of $\ell$, such that
$
\|B^{G_\ell}\|_{L^\infty}\leq   C_B
$ and
$
c_V\ell\leq \Vol(G_\ell)\leq  C_V\ell.
$
\end{remark}

\begin{remark}[Compatible stabilizer returns]\label{rmmon}
The construction of \cref{th16} can be modified to realize a compatible
factor symmetry.
          Let
              $d_j=|\mathsf H_{f_j}|$,
       $
k(\tau)=(e^{iq_0\tau},e^{-ip_0\tau}),
\,
K_0=\{k(\tau):\tau\in\RR\}.
$
Set
$D=\gcd(d_1q_0,d_2p_0)$.
Since
$\mathsf H_{f_j}=\langle e^{\frac {2\pi } {d_j} i}\rangle$,
\[
k(\tau)\in\mathsf H_{f_1}\times\mathsf H_{f_2}
\quad\Longleftrightarrow\quad
\tau\in\frac{2\pi}{D}\ZZ.
\]
Consequently,
$
(\mathsf H_{f_1}\times\mathsf H_{f_2})\cap K_0
=
\langle h\rangle,
\text{ where }
h:=k\left(\frac{2\pi}{D}\right)
$
and 
$
\ord(h)=D.
$
In particular,
a nontrivial compatible return exists exactly when $D>1$.

Assume $D>1$,
and let $\mathcal U$ be the low-energy interval chosen in the proof of
\cref{th16}.
Since $u$ is real analytic and nonconstant,
$u(\mathcal U)$ contains an open arc of $K_0$.
For every sufficiently large prime $\ell$,
the $\ell$-th roots of $h$ form a sufficiently fine mesh in $K_0$;
the finiteness of $\mathcal P$ therefore allows one to choose
$C_{2,\ell}\in\mathcal U$ such that,
with $u_\ell:=u(C_{2,\ell})$,
\[
u_\ell^\ell=h,
\qquad
\ord(u_\ell)=D\ell,
\qquad
\langle u_\ell\rangle\cap\mathcal P=\langle h\rangle.
\]
The choice of $\mathcal U$ excludes opposite-branch coincidences,
while the last identity shows that the first remaining coincidence
occurs after $\ell$ cells,
with phase $h$.
Hence
$m_{G,\ell}=\ell$ and $m_{\gamma,\ell}=D\ell$,
and the primitive spherical quotient is  embedded.
%gluing map $\Phi_h=\varphi_{h_1}\times\varphi_{h_2}$.
%Since $h\in K_0$,
%a constant phase rotation makes it special Legendrian
%in maximal contact dimension.
If the factors are special Legendrian,
then the quotient is minimal Legendrian.
Since $h\in K_0\subset\ker\chi_{p,q}$,
the gluing preserves its constant Legendrian angle,
so a constant phase rotation therefore makes the quotient special Legendrian.

Thus full phase saturation is unnecessary:
only the phase incidences met by $\langle u_\ell\rangle$
must be global stabilizers.
\Cref{expp} gives such a quotient with nonproduct topology.
\end{remark}

\begin{example}[Loss of embeddedness at the spherical stage]
\label{ex2}
     In Appendix \ref{apnp} we get an example of embedded special Legendrian  $M\cong     S^1\times  S^2\subset\Sphere^7$
    with $z_+=\lambda_*z_-$ for some $z_\pm\in M$ and $\lambda_*^8\ne1$.
    Two copies of $M$ and a steady contact profile give,
    up to a constant rotation,
    the special Legendrian map $F(\vartheta,z_1,z_2)=\frac{1 }{\sqrt 2} (e^{i\vartheta}z_1,e^{-i\vartheta}z_2)\in\Sphere^{15}$.
    For $\lambda_*=e^{i\vartheta_*}$,
    we then have $F(0,z_+,z_-)=F(\vartheta_*,z_-,z_+)$ with distinct tangent spaces.
    Thus embedded factors alone do not ensure embeddedness
    at the spherical stage.
    \end{example}
         
\begin{example}[Partial phase symmetry and nonproduct topology]
\label{expp}
     The same factor $M$ satisfies  $\mathsf H_M=\langle\zeta\rangle\subsetneq\mathcal P_M$, where $\zeta=e^{\frac \pi 4 i}$.
    Applying \cref{rmmon} to two copies of $M$,
    with $p=q=4$ and $h=(\zeta,\zeta^{-1})$, gives,
    for every sufficiently large prime $\ell$,
    an embedded special Legendrian $E_\ell\subset\Sphere^{15}$ with $m_{G,\ell}=\ell$ and $m_{\gamma,\ell}=8\ell$.
    Its underlying manifold is the mapping torus of
    $\Phi=\varphi_\zeta\times\varphi_{\zeta^{-1}}$.

    Since $\varphi_\zeta$ acts trivially on $H_1(M;\ZZ)$ and as $-1$ on $H_2(M;\ZZ)$,
    the K\"unneth formula and the mapping-torus sequence give $H_2(E_\ell;\ZZ)\cong\ZZ^3\oplus(\ZZ/2)^2$,
    whereas
                   $H_2(S^1\times M\times M;\ZZ)\cong\ZZ^5$.
    Thus $E_\ell$ is not even homotopy equivalent to the trivial product.
\end{example}

       For these same factors,
        a steady profile gives a crossing,
        whereas suitable nonsteady profiles give embedded quotients.
      The distinction lies in which phase incidences the profile realizes.

%\majorsectiongap
\section{Horizontal Lifts and Projective Delaunay Constructions}
          \label{se21}

          Every connected minimal Lagrangian immersion in projective space  admits a horizontal special Legendrian lift after passing to a canonical finite cover.
          If the original immersion is compact and embedded,
          then this lift is also embedded and phase-saturated.
          Consequently,
          for compact embedded inputs,
          every nonsteady contact profile gives a special Legendrian spherical product whose coincidences are all complete-cell returns.
          Whenever a spherical return exists,
          its primitive spherical quotient is therefore embedded.

          Hopf projection introduces a different obstruction:
          the primitive spherical quotient itself need not be phase-saturated,
          even though both factor lifts are.
          Its projective embeddedness will be measured by the reduced relative
          winding number.

      For the horizontal lifts below we have $k_i=n_i$.
       Throughout this section we retain  $p=n_1+1$, $q=n_2+1$, and set $m=p+q$.

   \subsection{Finite horizontal lifts and spherical quotients}

          Let $\pi_H:\Sphere^{2d+1}\to\CP^d$ be the Hopf fibration,
          with the Fubini--Study metric normalized to have holomorphic  sectional curvature $4$.
          For a Lagrangian immersion $\iota:L\to\CP^d$,
          the pullback of the Hopf connection is flat
          and
         we  denote its holonomy representation
          by
                $\chi_L:\pi_1(L)\to\Sphere^1$,
           the image by $\operatorname{im}\chi_L$.

           \begin{proposition}[Finite-holonomy lift] \label{pr13}
           Let $\iota:L^d\looparrowright\CP^d$ be a connected minimal Lagrangian immersion.
           Then $\operatorname{im}\chi_L$ is finite cyclic of order $h_L\mid2(d+1)$,
           and $h_L\mid d+1$ when $L$ is orientable.
           Passing to the cover $\pi_L:\widehat L\to L$ corresponding to $\ker\chi_L$, choose a horizontal lift $f:\widehat L\to\Sphere^{2d+1}$ satisfying $\pi_{\mathrm H}\circ f=\iota\circ\pi_L$.
           After a phase rotation, $f$ is special Legendrian.
           If $\iota$ is compact embedded,
           then $f$ is embedded and phase-saturated  with $\mathsf H_f=\operatorname{im}\chi_L$.
           In particular,
                               $\zeta^{d+1}\in\{\pm1\}$   for every $\zeta\in\mathsf H_f$.
       \end{proposition}

       \begin{proof}
                 On the universal cover,
           minimality makes the Legendrian angle of a horizontal lift constant.
           The Maslov--holonomy identity
          \cite{Oh1994,Reckziegel1985,Wang2005,Wolfson2004} gives
          $
                              \chi_L([\sigma])^{d+1}
                              \in\{\pm1\}
          $
          with sign $+1$ precisely for orientation-preserving loops.
           Hence $\operatorname{im}\chi_L$ has the stated order.
           The associated cover is therefore orientable
           and its global horizontal lift becomes special Legendrian after a  phase rotation.

           Suppose now that $\iota$ is embedded.
           If $f(y)=\zeta f(x)$,
           then Hopf projection shows that $x$ and $y$ differ by a deck
           transformation,
           and equivariance gives
           $\zeta\in\operatorname{im}\chi_L$.
           Thus $f$ is phase-saturated and embedded.
       \end{proof}

       \noindent\textit{Representative embedded minimal Lagrangians.}
          Beyond real projective spaces and Clifford tori,
          \cref{tb1} lists some compact embedded minimal Lagrangians
          and their canonical horizontal lifts.
          All entries are homogeneous.
          Parallel and nonparallel refer to the second fundamental form
          of the projective input.
          Each $K$ denotes the full projective isotropy group in its row,
          with identity component $K^0$.

      \begingroup
    \small
    \setlength{\tabcolsep}{3pt}
    \renewcommand{\arraystretch}{1.2}

\begin{longtable}{@{}
   >{\raggedright\arraybackslash}m{0.405\textwidth}
         >{\raggedright\arraybackslash}m{0.32\textwidth}
          >{\raggedright\arraybackslash}m{0.255\textwidth}@{}}

       \caption{Representative minimal Lagrangian embeddings
       and 
       their horizontal 
       lifts.
       }\label{tb1}\\

\toprule
    Projective input $L\subset\CP^d$
    & Horizontal lift $\widehat L\subset\Sphere^{2d+1}$
    & Class/source\\
    \midrule
    \endfirsthead

    \toprule
    Projective input $L\subset\CP^d$
    & Horizontal lift $\widehat L\subset\Sphere^{2d+1}$
    & Class/source\\
    \midrule
    \endhead

    \bottomrule
    \endfoot

          $\displaystyle \mathrm{PSU}(r)
            \hookrightarrow\CP^{r^2-1}$, $r\geq 2$
          & $\displaystyle \mathrm{SU}(r)
            \hookrightarrow\Sphere^{2r^2-1}$
          & Parallel;\newline
            Hamiltonian stable
            \cite{NaitohTakeuchi1982,AmarzayaOhnita2003}\\[0.45ex]

          $\displaystyle\frac{\mathrm{SU}(r)/\mathrm{SO}(r)}{\mathbb Z_r}
            \hookrightarrow\CP^{r(r+1)/2-1}$, $r\geq 3$
          & $\displaystyle
            \frac{\mathrm{SU}(r)}{\mathrm{SO}(r)}
            \hookrightarrow\Sphere^{r(r+1)-1}$
          & Parallel;\newline
            Hamiltonian stable
            \cite{NaitohTakeuchi1982,AmarzayaOhnita2003}\\[0.45ex]

          $\displaystyle
            \frac{\mathrm{SU}(2r)}{\mathrm{Sp}(r)\!\cdot\!\mathbb Z_{2r}}
            \hookrightarrow\CP^{r(2r-1)-1}$, $r\geq 3$
          & $\displaystyle
            \frac{\mathrm{SU}(2r)}{\mathrm{Sp}(r)}
            \hookrightarrow\Sphere^{2r(2r-1)-1}$
          & Parallel;\newline
            Hamiltonian stable
            \cite{NaitohTakeuchi1982,AmarzayaOhnita2003}\\[0.45ex]

          $\displaystyle \frac{E_6}{F_4\!\cdot\!\mathbb Z_3}
            \hookrightarrow\CP^{26}$
          & $\displaystyle \frac{E_6}{F_4}
            \hookrightarrow\Sphere^{53}$
          & Parallel;\newline
            Hamiltonian stable
            \cite{NaitohTakeuchi1982,AmarzayaOhnita2003}\\[0.45ex]

          $\displaystyle \frac{\mathrm{SU}(2)}{F}
            \hookrightarrow\CP^3$,
            $F\simeq\mathbb Z_3\rtimes\mathbb Z_4$
          & $\displaystyle \frac{\mathrm{SU}(2)}{\mathbb Z_3}
            \cong L(3,1)\hookrightarrow\Sphere^7$
          & Nonparallel
          %;
          %\newline
            %Hamiltonian stable
            \cite{Chiang2004}\\[0.45ex]

          $\displaystyle \mathrm{Sp}(3)/K\hookrightarrow\CP^{13}$,\newline
            $K^0=\mathrm{SU}(3)$,
            $K/K^0\simeq\mathbb Z_4$
          & $\displaystyle \frac{\mathrm{Sp}(3)}{\mathrm{SU}(3)}
            \hookrightarrow\Sphere^{27}$
          & Nonparallel
            \cite{BedulliGori2008}\\[0.45ex]

          $\displaystyle \mathrm{SU}(6)/K\hookrightarrow\CP^{19}$,\newline
            $K^0=\mathrm{SU}(3)\times\mathrm{SU}(3)$,
            $K/K^0\simeq\mathbb Z_4$
          & $\displaystyle
            \frac{\mathrm{SU}(6)}{\mathrm{SU}(3)\times\mathrm{SU}(3)}
            \hookrightarrow\Sphere^{39}$
          & Nonparallel
            \cite{BedulliGori2008}\\[0.45ex]

          $\displaystyle
            \frac{\mathrm{SU}(7)}{G_2\!\cdot\!\mathbb Z_7}
            \hookrightarrow\CP^{34}$
          & $\displaystyle \frac{\mathrm{SU}(7)}{G_2}
            \hookrightarrow\Sphere^{69}$
          & Nonparallel
            \cite{BedulliGori2008}\\[0.45ex]

          $\displaystyle \mathrm{Spin}(12)/K\hookrightarrow\CP^{31}$,\newline
            $K^0=\mathrm{SU}(6)$, $K/K^0\simeq\mathbb Z_4$
          & $\displaystyle
            \frac{\mathrm{Spin}(12)}{\mathrm{SU}(6)}
            \hookrightarrow\Sphere^{63}$
          & Nonparallel
            \cite{BedulliGori2008}\\[0.45ex]

          $\mathrm{SU}(8)/K\hookrightarrow\CP^{55}$,\newline
            $K^0=\mathrm{PSU}(3)$, $K/K^0\simeq\mathbb Z_{16}$
          & $\displaystyle \frac{\mathrm{SU}(8)}{\mathrm{PSU}(3)}
            \hookrightarrow\Sphere^{111}$
          & Nonparallel
            \cite{BedulliGori2008}\\[0.45ex]

          $\mathrm{Spin}(14)/K\hookrightarrow\CP^{63}$,\newline
            $K^0=G_2\times G_2$, $K/K^0\simeq\mathbb Z_8$
          & $\displaystyle \frac{\mathrm{Spin}(14)}{G_2\times G_2}
            \hookrightarrow\Sphere^{127}$
          & Nonparallel
            \cite{BedulliGori2008}\\[0.45ex]

\end{longtable}
\endgroup

          In the symmetric-matrix row,
          $\mathbb Z_r$ acts by scalar multiplication on the standard matrix model
          of $\mathrm{SU}(r)/\mathrm{SO}(r)$.
         % The skew-matrix orbit 
         The lift in the third row has different horizontal presentation
          $\mathrm{SU}(2r-1)/\mathrm{Sp}(r-1)$;
          for $r=3$ this is $\mathrm{SU}(5)/\mathrm{Sp}(2)$.

           \begin{theorem}[Spherical return classification]
           \label{th17}
           Let $\iota_i:L_i^{n_i}\hookrightarrow\CP^{n_i}$,
           $i=1,2$,
           be compact connected embedded minimal Lagrangians,
           $f_i:\widehat L_i\hookrightarrow\Sphere^{2n_i+1}$
           their canonical horizontal lifts,
           and $\gamma$  a nonsteady contact profile.
           After a  phase rotation,
           $G_\gamma$ is special Legendrian.

           Every coincidence of $G_\gamma$ satisfies $t'-t=kT_{\mathrm{cell}}$
            for some $k\in\mathcal R_G$.
           Consequently,
           if $\mathcal R_G=\{0\}$,
           then $G_\gamma$ is injective but not an embedding.
           If $\mathcal R_G=m_G\ZZ\ne\{0\}$,
           then,
           up to a phase rotation,
               $G_\gamma$ descends to a compact embedded special Legendrian
                         $ \overline G_\gamma:
              X_\gamma
              \hookrightarrow
              \Sphere^{2n_1+2n_2+3},$
called its primitive spherical quotient.
Here $X_\gamma:=\mathcal M_G$ is the mapping torus of \cref{th6}, and $\overline G_\gamma$ is its descended map.
       \end{theorem}

       \begin{proof}
           Suppose
           $G_\gamma(t,x_1,x_2)=G_\gamma(t',x_1',x_2')$.
           Then $s(t)=s(t')$ and
           $h_i:=e^{i(\theta_i(t)-\theta_i(t'))}\in\mathsf H_{f_i}$.
           With %$p=n_1+1$, $q=n_2+1$, and
           $\Psi=p\theta_1+q\theta_2$,
           by \cref{pr13} we have
           \[
                              e^{i(\Psi(t)-\Psi(t'))}
                              =h_1^ph_2^q\in\{\pm1\}.
           \]
           \cref{lm12} therefore excludes oppositely oriented interior
           branches.
           As a result,
           $t'-t=kT_{\mathrm{cell}}$ for some $k\in\ZZ$.
           The same holds at a turning point.
           Thus $e^{ik\Delta_i}\in\mathsf H_{f_i}$ and
           $k\in\mathcal R_G$.
           Conversely,
           every $k\in\mathcal R_G$ gives such an identification.

           If $\mathcal R_G=\{0\}$,
            then the immersion is injective,
           and it is easy to see that
           it cannot be an embedding.
           When $\mathcal R_G=m_G\ZZ\ne\{0\}$,
           the primitive spherical quotient is embedded.
           By \cref{pr12},
           it is always,
           up to a phase rotation,
                     special Legendrian.
       \end{proof}

%   The preceding theorem gives an embedded primitive spherical quotient $X_\gamma\subset\Sphere^{2n_1+2n_2+3}$.
%      We now pass to its Hopf quotient.
 %       Since the Hopf map ignores a common phase,
%        the first projective return may occur before the first spherical return. 
 %  The next subsection determines this projective return and the embeddedness of the resulting projective quotient.

Assume henceforth that
    $\mathcal R_G=m_G\ZZ\ne\{0\}$ and retain
       $X_\gamma$ and $\overline G_\gamma$ from \cref{th17}.
       Let $\mathsf H_{G_\gamma}$ be the phase stabilizer of $\overline G_\gamma(X_\gamma)$.
            The Hopf projection of $\overline G_\gamma$ descends to
            \begin{equation}\label{DaXg}
                \overline \iota_\gamma:   \mathcal X_\gamma :=    X_\gamma/\mathsf H_{G_\gamma} 
                \longrightarrow
                           \CP^{n_1+n_2+1}.
            \end{equation}
            Let $q_\gamma:X_\gamma\to\mathcal X_\gamma$ stand for the quotient map.
            Then
            this descent is characterized by
                 $\overline\iota_\gamma\circ q_\gamma =\pi_{\mathrm H}\circ\overline G_\gamma$.
            We call $\overline \iota_\gamma$ the projective Delaunay
            immersion associated with $(\iota_1,\iota_2,\gamma)$.
            As additional Hopf identifications may occur,
            $\overline \iota_\gamma$ need not be embedded.

            \begin{remark}[Block homogeneous coordinates] 
            \label{rm:HC} 
           For $[x]\in\mathcal X_\gamma$, choose a representative
           $x\in X_\gamma$ and
                write $\overline G_\gamma (x)=(Z_1(x),Z_2(x))$ with respect to
            $\CC^{n_1+n_2+2}=\CC^{n_1+1}\oplus\CC^{n_2+1}$.  
            Then
            $$
              \overline \iota_\gamma([x])
                                  =[Z_1(x):Z_2(x)]
                                                     =[(\cos s)e^{i\theta_1}f_1:(\sin s)e^{i\theta_2}f_2]
                                                     $$
                                                     with
            $  |Z_1|^2=\cos^2s,\, |Z_2|^2=\sin^2s$.
            The common phase ambiguity leaves the block norms invariant,
            so functions  $  |Z_1|^2,\, |Z_2|^2$ are globally defined on
            $\mathcal X_\gamma$.
               Since both blocks are nonzero,
               blockwise projectivization gives a well-defined map
               $
               \mathcal X_\gamma \to L_1\times L_2,
               $
               $[Z_1: Z_2]\mapsto ([Z_1], [Z_2])$.
                        \end{remark}

   \subsection{Hopf descent and projective factor slices}\label{sub62}

       \begin{proposition}[Hopf descent criterion]
           \label{pr14}
           Let $F:X^d\hookrightarrow\Sphere^{2d+1}$ be a compact connected
            Legendrian embedding,
             with phase stabilizer $\mathsf H_F$.
           Then $\mathsf H_F$ is finite cyclic and acts freely on $X$.
           The Hopf map induces an immersion
           $\overline F:X/\mathsf H_F\to\CP^d$,
           which is embedded if and only if $F(X)$ is phase-saturated.
           If $F$ is minimal Legendrian,
           then $\overline F$ is minimal Lagrangian.
       \end{proposition}

       \begin{proof}
    By \cref{lm9} it follows that $\mathsf H_F$ is finite cyclic,  acting on $X$ freely.
    Since $F$ is horizontal, 
    $\pi_H\circ F$ induces an immersion $\overline F:X/\mathsf H_F\to\CP^d$.
    The equivalence follows directly from the definition of phase saturation.
        The last assertion is the standard Hopf correspondence between minimal Legendrian immersions and minimal Lagrangian immersions; e.g., see \cite{CastroLiUrbano2006,Reckziegel1985}.
       \end{proof}

        %   Assume from now on $\mathcal R_G=m_G\ZZ\ne\{0\}$
         %  and retain
         %  $X_\gamma,      \overline G_\gamma,        \mathcal X_\gamma$
             %  and
            %       $\overline \iota_\gamma$
              % from above.
      Now let us introduce
\[
%\, \, \,\, 
 \mathsf K_{12}
   :=\mathsf H_{f_1}\mathsf H_{f_2}^{-1},
   \qquad
   \kappa:=|\mathsf K_{12}|
   =\operatorname{lcm}\bigl(|\mathsf H_{f_1}|,|\mathsf H_{f_2}|\bigr),
   \qquad
%   \Lambda:=\frac{2\pi}{\kappa},
%   \qquad 
   \delta:=\Delta_1-\Delta_2>0.
\]
            Thus $\mathsf K_{12}=\langle  e^{\frac{2\pi i}{\kappa}}  \rangle$,
            collects relative phases that can be absorbed by the two factor stabilizers.
               For $e^{im_G\delta}\in\mathsf K_{12}$,
         we can write
         \[
         \frac{\kappa\delta}{2\pi} 
         =
                 \frac{A}{N},   \qquad  A,N\in\NN,   \qquad   \gcd(A,N)=1.
            \]
We call $A$ the \emph{reduced relative winding number}.
%Actually this $N$ also has meanings.
%
%
%
            As Hopf projection eliminates common phase, we define the \emph{projective return group} by
\[
   \mathcal R_L
   :=
   \left\{
   r\in\ZZ:
   e^{ir\Delta_i}
   =
   \lambda h_i,
   \,
   \text{ for some }
   \lambda\in\Sphere^1,
   \,
   h_i\in\mathsf H_{f_i},
\,
   i=1,2
   \right\}.
\]
%I.e., $  \mathcal R_L  =   \left\{   r\in\ZZ:    e^{ir\delta}\in\mathsf K_{12} \right\}. $
Its positive generator    is called   the    \emph{projective return order}.

\begin{proposition}[Projective return order]
\label{prPR}
The projective return group  $ \mathcal R_L=N\ZZ $.
Consequently,
$N$ is the projective return order,
and
$
   N\mid m_G\mid m_\gamma.
$
\end{proposition}

\begin{proof}
By the definition of $\mathcal R_L$ and $\mathsf K_{12}=\langle  e^{\frac{2\pi i}{\kappa}}  \rangle$,
\begin{equation}\label{eqRL}
   r\in\mathcal R_L
   \quad\Longleftrightarrow\quad
   e^{ir\delta}\in\mathsf K_{12}
 \quad\Longleftrightarrow\quad
    \frac{\kappa r\delta}{2\pi}
       =
   \frac{rA}{N}
    \in\ZZ.
\end{equation}
Hence
$
   N\mid r
$
and 
$\mathcal R_L=N\ZZ$
with
$
   m_\gamma\ZZ
   \subset
   \mathcal R_G
   =
   m_G\ZZ
   \subset
   \mathcal R_L
   =
   N\ZZ.
$
\end{proof}
       We denote by $\mathcal C_\gamma$ the virtual  $N$-cell base circle 
            obtained by quotienting the profile  by the translation
                  $t\mapsto t+NT_{\mathrm{cell}}$ 
                  underlying one primitive projective return, parametrized by projective arclength $\tau$.
                             The proposition determines all  complete-cell projective returns.
                            However, incoming and outgoing slices may meet without forming a global return.  
                             To detect these incidences, 
                              set
\[
   \mathsf D_{12}:=\mathsf H_{f_1}\cap\mathsf H_{f_2}
\]
acting diagonally on each fixed-time slice
and
$
   \mathcal F_t
   :=
   \pi_H\bigl(
   G_\gamma(
   \{t\}\times\widehat L_1\times\widehat L_2)
   \bigr)
   \cong
   (\widehat L_1\times\widehat L_2)/\mathsf D_{12}.
$

\begin{theorem}[Projective self-intersections]
\label{pr15}
Under the hypotheses and notation above,
$\overline \iota_\gamma$ is embedded
if and only if $A=1$.
When $A\geq 2$,
the self-intersection locus of $\overline \iota_\gamma$
is the disjoint union of exactly
$N(A-1)$ projective factor slices.
Along each such slice exactly two local sheets meet, sharing codimension-one tangent space corresponding to the slice.
\end{theorem}

\begin{proof}
With $\psi:=\theta_1-\theta_2$,
phase saturation of the factors gives the incidence criterion
\[
   \mathcal F_t\cap\mathcal F_{t'}\ne\varnothing
   \quad\Longleftrightarrow\quad
   \mathcal F_t=\mathcal F_{t'}
   \quad\Longleftrightarrow\quad
   s(t)=s(t')
   \ \text{and}\
   e^{i(\psi(t)-\psi(t'))}\in\mathsf K_{12}.
\]
This criterion alone does not imply a smooth return.

For $s\in[s_-,s_+]$,
let $t_-(s)$ and $t_+(s)$ be the incoming and outgoing points
adjacent to a fixed lower turning point,
and set
\[
   B(s)
   :=
   \psi(t_+(s))-\psi(t_-(s))
   =
   2\int_{s_-}^{s}
   \frac{\tan u+\cot u}
        {\sqrt{D_{C_2}(u)}}\,\dd u.
\]
Then $B$ increases strictly from $0$ to $\delta$
and every opposite-branch phase difference is congruent to
$B(s)$ modulo $\delta$.
          % Since $B'(s)>0$ in the interior,
        Thus the incoming and outgoing profile directions are distinct modulo the tangent space of the common factor slice.  
       Consequently, every opposite-branch incidence is a crossing rather than a smooth return. 
         Therefore, every smooth projective return is a complete-cell shift determined by
        $\mathcal R_L$.
        Two opposite branches separated by $r$ complete cells meet
        projectively precisely when
                  \begin{equation}\label{eqPI}
                      e^{(B(s)+r\delta)i}\in \mathsf K_{12}.
                   \end{equation}
          Set $\Lambda=\frac {2\pi} \kappa$.  
          Requirement \eqref{eqPI} is equivalent to $B(s)+r\delta\in\Lambda\ZZ$.  
          Since $\delta=\frac {A\Lambda} N$ and $\gcd(A,N)=1$, 
                 such an $r$ exists exactly when $B(s)\in \frac {\Lambda} N \ZZ$.  
                 Thus the interior incidence levels are precisely
                        $B(s_\nu)= \frac {\nu\Lambda} N $, for $\nu=1, \, \ldots,A-1$.
               For each $\nu$, 
                   every one of the $N$ cells in a primitive projective period contributes one self-intersecting slice;
                    its cell difference $r_\nu$ is uniquely determined modulo $N$ by $\nu+r_\nu A\equiv0\pmod N$.  
              Strict monotonicity makes the levels $s_\nu$ distinct, 
                        while uniqueness of $r_\nu$ gives each sheet a unique opposite-branch partner.  
                        Hence there are exactly $N(A-1)$ components, 
                        with exactly two local sheets along each one. 
                 By the branch-direction observation above, these are the stated crossings.

     Finally, it is clear that $\overline \iota_\gamma$ is embedded exactly when $A=1$.
\end{proof}

      %    Next corollary tells us that there are resulting embeddings.

       \begin{corollary}[Embedded projective resonance]
           \label{co7}
           For the canonical lifts fixed above,
                      if $p+q>2$ and
$
   \sqrt{\frac{2pq}{p+q}}
   <
   \kappa N
   <
   \frac{2pq}{p+q}
$
for some $N\in\NN$,
           then there is a regular ordinarily closed contact profile with                       $ \frac{\kappa\delta}{2\pi}=\frac1N$.
                                         Its primitive spherical quotient is embedded special
           Legendrian,
           and its Hopf quotient is a compact embedded Delaunay-type minimal
           Lagrangian with projective return order $N$.
       \end{corollary}

       \begin{proof}
           By \cref{lm15,lm16},
           $
                              \frac{\delta}{2\pi}
                              \rightarrow
                              \sqrt{\frac{p+q}{2pq}}
                               \quad(C_2\downarrow P_*),
                         \,
                              \frac{\delta}{2\pi}
                              \rightarrow
                              \frac{p+q}{2pq}
                               \quad(C_2\uparrow\infty).
           $
           The hypothesis places $\frac{1}{\kappa N}$ between these limits,
           so the intermediate value theorem asserts 
           the existence of $C_2$ such that
           $\frac{\kappa\delta}{2\pi}=\frac{1}{N}$.
           Equation \eqref{eq74} then gives
           $\frac{\Delta_1}{2\pi}=\frac{q}{\kappa N(p+q)}$ and
           $\frac{\Delta_2}{2\pi}=-\frac{p}{\kappa N(p+q)}$.
           Hence the profile closes ordinarily.
           Since the winding number $A=1$ now,
the conclusions follow from
\cref{th17,pr14,pr15}.
       \end{proof}

                For the standard embeddings  $\mathbb{RP}^{3}\subset\CP^{3}$ and $\mathbb{RP}^{5}\subset\CP^{5}$, the canonical lifts have $(p,q,\kappa)=(4,6,2)$.
       Since $\sqrt{\frac{24}{5}}<4<\frac{24}{5}$,
              \cref{co7} applies with $N=2$
               and yields a compact embedded Delaunay-type minimal Lagrangian in $\CP^{9}$.
       Here $(\frac{\Delta_1}{2\pi}, \frac{\Delta_2}{2\pi})=(\frac{3}{20},-\frac{1}{10})$, and $(m_\gamma, m_G, N, A)=(20, 10, 2, 1)$.
       In the block homogeneous coordinates of \cref{rm:HC},
       the covering in \cref{lm19} has two sheets, distinguished by the relative sign of the blocks.
           The fiber over $\mathcal C_\gamma$ is
           $\mathcal F=(S^3\times S^5)/\ZZ_2$, with diagonal antipodal action,
           and
                  its return map $[x,y]\mapsto[-x,y]$ is isotopic to the identity through rotations of $\Sphere^3$,
               so the domain is diffeomorphic to $S^1\times\mathcal F$.           

          \begin{remark}
          In the symmetric case $p=q=r$,
          the hypothesis is $\sqrt r<\kappa N<r$.
          If $\kappa=1$ and $r\geq 3$,
          one may take $N=r-1$.
         % and 
         % the corresponding energy.
          \end{remark}

   \subsection{Hamiltonian index inheritance}
          \label{ss9}
                
                For a Lagrangian immersion,
                Hamiltonian variations are the exact part of the infinitesimal Lagrangian variations:
                  $\alpha_V:=\omega(\,\cdot\,,V)=\dd v$,
                  or equivalently $V=J\nabla v$. 
                  For a compact connected minimal Lagrangian immersion
                         $F:X^d\looparrowright\CP^d$,
                   let $\Ind_{\mathrm{Ham}}(F)$ denote the index of $v\mapsto Q^{\mathrm{Ham}}_F(J\nabla v)$ on $C^\infty(X)/\RR$,
                identified with the mean-zero functions.
              For a self-intersecting immersion, this index is understood on its domain.

     By Oh's formula \cite{Oh1990},
     the Hamiltonian Jacobi operator is
        $
             \Delta_0\bigl(\Delta_0-2(d+1)\bigr),
         $
       where $\Delta_0$ is the positive Laplacian.
       Hence the Hamiltonian index counts, with multiplicity, the eigenvalues of $\Delta_0$ in $(0,2(d+1))$.
          For $\overline \iota_\gamma$,
          the profile functions form an invariant sector on which $\Delta_0$ reduces to a periodic scalar Sturm--Liouville operator.
       Two further factor sectors are constructed in  Appendix \ref{ap11}.

       \begin{theorem}[Hamiltonian index inheritance]
           \label{th18}
           Under the hypotheses of \cref{pr15},
            we have
           \begin{equation}
                              \label{eq84}
                              \Ind_{\mathrm{Ham}}(\overline \iota_\gamma)
                              \geq 
                              \Ind_{\mathrm{Ham}}(\iota_1)
                               +
                              \Ind_{\mathrm{Ham}}(\iota_2)
                               +
                              2N-2.
           \end{equation}
       \end{theorem}

       \begin{proof}
        Set $m=p+q$.
        Recall the projective quotient $\overline \iota_\gamma : \mathcal X_\gamma\longrightarrow  \CP^{m-1}$ in \eqref{DaXg}.
        Let $\Delta_0$ be the positive Laplacian on $\mathcal X_\gamma$
        and let $\tau$ be the arclength coordinate on $\mathcal C_\gamma$.
         By Oh's formula \cite{Oh1990},
           \[
                              Q_{\overline \iota_\gamma}^{\mathrm{Ham}}(v,w)
                               =
                              \left\langle
                              \Delta_0(\Delta_0-2m)v,w
                              \right\rangle_{L^2}.
           \]
      %     Thus the negative directions correspond to eigenvalues
        %   $0<\lambda(\Delta_0)<2m$.
  %
     %      
          On the profile sector,
          the warped-product formula gives
     \begin{equation}\label{SLop}
          L_{\mathrm{prof}}  = -\frac1\varpi\frac{\dd}{\dd\tau}
          \left(\varpi\frac{\dd}{\dd\tau}\right),
     \qquad
           \varpi=(\cos s)^{n_1}(\sin s)^{n_2},
       \end{equation}
         namely,
         for every profile function $v=v(\tau)$,
         one has
         $   \Delta_0v=L_{\mathrm{prof}}v $.
Now 
         Oh's formula restricts to
         $
              Q_{\overline \iota_\gamma}^{\mathrm{Ham}}(v)
              =
              c_0\big\langle v,
              L_{\mathrm{prof}}(L_{\mathrm{prof}}-2m)v
              \big\rangle_{L^2(\varpi\,\dd\tau)},
              \text{ for some } c_0>0.
        $

           The key point of   \eqref{SLop}  is that 
          profile contribution is a periodic scalar Sturm--Liouville problem.
          The negative modes are precisely the eigenfunctions of
           $L_{\mathrm{prof}}$ with eigenvalues in $(0,2m)$.
             By \cref{rm:HC},
              $|Z_1|^2=\cos^2s$ is globally defined on $\mathcal X_\gamma$.
              After pulling back to $X_\gamma$,
              the Takahashi identity \cite{Takahashi1966} and
              $|\nabla Z_1|^2=n_1+\sin^2s$ give
              $\Delta_0|Z_1|^2=2m|Z_1|^2-2p$.
              Since $X_\gamma\to\mathcal X_\gamma$ is locally isometric,
              this identity descends to $\mathcal X_\gamma$.
              Consequently,
           \begin{equation}
                              \label{eq85}
                              u=m|Z_1|^2-p=m\cos^2s-p,
                              \qquad L_{\mathrm{prof}}u=2m\,u.
           \end{equation}
           Since $s_-<s_*<s_+$ and $\cos^2s_*=\frac{p} m$,
           the function $u$ has exactly two simple zeros in each complete cell.
           The primitive projective return occurs after $N$ complete cells by
           \cref{prPR},
           so, if
           $0=\lambda_0\leq \lambda_1\leq \cdots$
           is the periodic spectrum of $L_{\mathrm{prof}}$
           over the virtual projective circle $\mathcal C_\gamma$,
           Sturm oscillation \cite{Eastham1973} places $2m$ at $\lambda_{2N-1}$ or $\lambda_{2N}$.
           After removing the constant mode,
           the profile sector contributes at least $2N-2$ negative
           directions.
           Appendix~\ref{ap11} identifies two further mutually orthogonal
           sectors with the Hamiltonian Hessians of $\iota_1$ and $\iota_2$,
           up to positive constants.
           Adding the negative subspaces in these three sectors proves
           \eqref{eq84}.
       \end{proof}

   %    \begin{remark}[Lagrangian index inequality]
  %         \label{co9}
  %     %
  %             Hodge decomposition and Oh's formula \cite{Oh1990} lead to
  %    $ \Ind_{\mathrm{Lag}}(\overline \iota_\gamma)         =      
   %                         \Ind_{\mathrm{Ham}}(\overline \iota_\gamma)
    %                           +
    %                          b_1(\mathcal X_\gamma)$.
    %       Together with \eqref{eq83} and \cref{th18}, this gives
    %    $
   %                           \Ind_{\mathrm{Lag}}(\overline \iota_\gamma)
  %                            \geq 
   %                           \Ind_{\mathrm{Lag}}(\iota_1)
    %                           +
    %                          \Ind_{\mathrm{Lag}}(\iota_2)
   %                            +
   %%                           2N-1.
   %    $
    %   \end{remark}

\begin{corollary}[Lagrangian index inheritance]
           \label{co9}
           Let $\Ind_{\mathrm{Lag}}$ denote the index of the Lagrangian  second variation on closed $1$-forms.  
           Under the hypotheses of \cref{pr15},
           \[
              \Ind_{\mathrm{Lag}}(\overline\iota_\gamma)
              \geq 
              \Ind_{\mathrm{Lag}}(\iota_1)
              +\Ind_{\mathrm{Lag}}(\iota_2)+2N-1.
           \]
       \end{corollary}

       \begin{proof}
           For a compact minimal Lagrangian immersion
           $F:Y^d\looparrowright\CP^d$, Hodge decomposition and Oh's
           formula \cite{Oh1990} give
          $
              \Ind_{\mathrm{Lag}}(F)
              =\Ind_{\mathrm{Ham}}(F)+b_1(Y).
          $
           Indeed, the exact and harmonic sectors are orthogonal, 
           and the quadratic form on harmonic $1$-forms is
                   $-2(d+1)\|  \, \cdot \,  \|_{L^2}^2$. 
            Apply this identity to  $\overline\iota_\gamma,\iota_1,\iota_2$ and employ \eqref{eq83} and \cref{th18}.
       \end{proof}

       \begin{example}[Loss of embeddedness under Hopf projection]
           \label{ex1}
        Let $f_1=f_2=f$ be the Harvey--Lawson special Legendrian torus in $\Sphere^5$,
           whose Hopf projection is the Clifford torus in $\CP^2$.
           The restricted Hopf map is its canonical threefold holonomy cover,
           with deck group $\mathsf H_f\cong\ZZ_3$.
                      For $p=q=3$,
            \cref{lm15,lm16} and continuity give a regular profile 
           with  $(\Delta_1,\Delta_2)=(\frac{2\pi}{5},-\frac{2\pi}{5})$.
           Its primitive spherical quotient is embedded special Legendrian in
           $\Sphere^{11}$ and $m_\gamma=m_G=N=5$.
           It follows that $(A,N)=(6,5)$.
           By \cref{pr15},
           the Hopf image in $\CP^5$ has exactly $25$ clean self-intersection $4$-tori.
           Since $b_1(\TT^2)=2$,
           \cref{th18,co9}  lead to
           $\Ind_{\mathrm{Ham}}(\overline \iota_\gamma)\geq 8$ and
           $\Ind_{\mathrm{Lag}}(\overline \iota_\gamma)\geq 13$.
       \end{example}

          \appendix

%%\majorsectiongap
\section{Period and Phase Integrals}
          \label{ap1}

          Here we collect the analytic and asymptotic integral facts used for the
          phase map, endpoint anchors, and closing arguments.

   \subsection{Fixed-domain analyticity of period integrals}
% The following lemma gives the analyticity used for the phase and radial
%action integrals.

       \begin{lemma}[Analyticity of two-turning integrals]
           \label{lm20}
           Let $D(s,\boldsymbol\lambda)$ and $\cN(s,\boldsymbol\lambda)$
            be real analytic near
            $[s_-^0,s_+^0]\times\{\boldsymbol\lambda_0\}$,
            where $0<s_-^0<s_+^0<\frac{\pi}{2}$.
           Assume that $s_\pm^0$ are simple zeros of
            $D(\,\cdot\,,\boldsymbol\lambda_0)$,
            and that $D(\,\cdot\,,\boldsymbol\lambda_0)>0$
            between them.
           Then the nearby roots $s_\pm(\boldsymbol\lambda)$,
            and
           \[
                              I(\boldsymbol\lambda)
                               =
                              \int_{s_-(\boldsymbol\lambda)}^{s_+(\boldsymbol\lambda)}
                              \frac{\cN(s,\boldsymbol\lambda)}
                              {\sqrt{D(s,\boldsymbol\lambda)}}\,\dd s,
           \]
           depend real analytically on $\boldsymbol\lambda$.
       \end{lemma}

       \begin{proof}
           The implicit-function theorem gives analytic roots
            $s_\pm(\boldsymbol\lambda)$.
           Analytic division gives
                 $
                              D(s,\boldsymbol\lambda) = (s-s_-)(s_+-s)K(s,\boldsymbol\lambda) \,  \text{with } K>0.
                             $
           With
          $
                              s
                               =
                              s_-
                               +
                              (s_+-s_-)\sin^2u,
                               \,
                                0\leq  u\leq \frac{\pi}{2},
          $
           the endpoint factors cancel,
            and
           \[
                              I(\boldsymbol\lambda)
                               =
                              2\int_0^{\frac{\pi}{2}}
                              \frac{\cN(s(u,\boldsymbol\lambda),\boldsymbol\lambda)}
                              {\sqrt{K(s(u,\boldsymbol\lambda),\boldsymbol\lambda)}}\,\dd u.
           \]
           Note that the integrand is jointly real analytic on a fixed neighborhood of
            $[0,\frac{\pi}{2}]\times\{\boldsymbol\lambda_0\}$.
            Hence $I$ is real analytic.
       \end{proof}

   \subsection{Adjusted phase lift at $C_1=0$}
          \label{ap2}

       \begin{proof}[Proof of \cref{lm6}]
         Use $
                              p=k_1+1,
                              \, 
                              c=C_1,
                              \,
                               E=C_2,
                            $
and $\varepsilon=\varepsilon_1$.
   % Since the degenerate endpoint is $b=\sin s=0$,
     Set $y=\frac{b^2}{a^2}=\tan^2s$.
     Then       $
      \dd s=\frac{\dd y}{2\sqrt y(1+y)}
      $.
    By % introducing 
    $\mathcal F_E(y):=y(E(1+y)^{-p}-1)$,
    we get
        $
      D_{c,E}(s)
         =\frac{\mathcal F_E(y)-c^2}{1+y}
         $,
      and
      \[
      \Delta_1
      =\varepsilon\int_\alpha^\beta
        \frac{\dd y}
        {\sqrt{1+y}\sqrt{\mathcal F_E(y)-c^2}},
        \qquad
      \Delta_2
      =\varepsilon c\int_\alpha^\beta
        \frac{\dd y}
        {y\sqrt{1+y}\sqrt{\mathcal F_E(y)-c^2}},
    \]
    where $\alpha<\beta$ satisfy $\mathcal F_E(y)=c^2$.
    At $c=0$,  we have $\alpha=0$ and $\beta=E^{\frac 1 p}-1$.
          Since $E>1$,
          it follows that
          $\mathcal F_E'(0)=E-1>0$ 
          and
          $\mathcal F_E'(\beta)=- \frac {p\beta} {1+\beta}<0$.
         As both roots  %$\alpha$ and  $\beta$ 
          are simple,
    by  implicit function theorem they are analytic in $(c^2,E)$.
    Analytic division gives
    \[
      \begin{aligned}
      \mathcal F_E(y)-c^2
         &=(y-\alpha)(\beta-y)K(y,c^2,E),
         \quad &K&>0,\\
      y_\vartheta
         &=\alpha+(\beta-\alpha)\sin^2\vartheta,
         &h&=[(1+y)K]^{- \frac 1 2}=h_0+yh_1,
      \end{aligned}
    \]
    where $h_0=h(0,c^2,E)$.
    Here $K,h_0,h_1$ are analytic,
    $\alpha\beta K(0,c^2,E)=c^2$,
    and
    $\int_0^{\frac \pi 2} y_\vartheta^{-1}\dd\vartheta
      =\pi (2\sqrt{\alpha\beta})^{-1}$.
    Therefore $y=y_\vartheta$ gives, for $c\ne0$,
    \[
      \begin{aligned}
      \Delta_1
         =2\varepsilon\int_0^{\frac\pi2}
           h(y_\vartheta,c^2,E)\,\dd\vartheta,
           \,
            \qquad
      & \Delta_2
         =\varepsilon\pi\sgn c
           +2\varepsilon c\int_0^{\frac\pi2}
           h_1(y_\vartheta,c^2,E)\,\dd\vartheta,\\
      &\Delta_2^\sharp(c,E)
         =\varepsilon\pi
           +2\varepsilon c\int_0^{\frac\pi2}
           h_1(y_\vartheta,c^2,E)\,\dd\vartheta.
      \end{aligned}
    \]
    The first formula gives the analytic extension of $\Delta_1$
    and agrees at $c=0$ with the value prescribed in the statement.
    The second gives the one-sided limits,
    while the third exhibits  the analytic extension of
    adjusted $\Delta_2^\sharp$.

    When $c=0$, 
          we have $\alpha=0$ and $\beta=E^{\frac 1 p}-1$.
          Hence $E=(1+\beta)^p$.
           Therefore,
           \[
             %\begin{aligned}
             \mathcal F_E(y)
             %&=
            =
             {y\bigl(E-(1+y)^p\bigr)}{(1+y)^{-p}}
             =
         %    {y\bigl((1+\beta)^p-(1+y)^p\bigr)}
          %        {(1+y)^{-p}}
           %       \\
           % & =
             y(\beta-y)
         %    {(1+y)}
           \,  {
             \sum_{j=0}^{p-1}
             (1+\beta)^{p-1-j}(1+y)^{j-p}
             }
             %.
            % \end{aligned}
           \]
           %Comparing this with
          As $\mathcal F_E(y)=y(\beta-y)K(y,0,E)$,
          we get
          $
          (1+y)   K(y,0,E)
             =
               \sum_{j=0}^{p-1}
             (1+\beta)^{p-1-j}(1+y)^{j-p+1}.
         $
           Consequently,
          $
             h(y,0,E)^2
             =
             \frac{1}{(1+y)K(y,0,E)}
             =
             \big[
             \sum_{j=0}^{p-1}
             \big(
             \frac{E^{\frac 1 p}}{1+y}
             \big)^j
             \big]^{-1}.
          $
           If $p>1$, the expression is strictly increasing in $y$.
           Hence
          $
             h_1(y,0,E)
             =
             \frac{h(y,0,E)-h(0,0,E)}{y}>0
              \text{ for }y>0.
         $
           By differentiating 
           $\Delta_2^\sharp$ 
           with respect to $c$
           at $c=0$ 
           we get
                    $
             \varepsilon\,
             \partial_c\Delta_2^\sharp(0,E)
             =
             2\int_0^{\frac\pi2}
             h_1(y_\vartheta,0,E)\,\dd\vartheta>0.
          $
           
                 If $p=1$,
    direct integration gives
    $\Delta_1=\Delta_2^\sharp=\varepsilon\pi$.
              \end{proof}

   \subsection{Endpoint anchors}

         The proofs of \cref{lm8,lm15,lm16} use two limits, with $E=C_2$ in the applications:
          coalescing turning points at the lowest energy,
          and endpoint rescaling at high energy.
          First, 
          let
          $A$ be positive and real analytic near a nondegenerate maximum $s_*$,
          set $E_*=A(s_*)^{-1}$ and $\kappa=-\tfrac12(\log A)''(s_*)>0$, and let $s_-(E)<s_+(E)$ be the adjacent roots of $EA(s)=1$.
          For every continuous $w$ near $s_*$,
          \begin{equation}\label{eqath}
             \int_{s_-(E)}^{s_+(E)}
                \frac{w(s)}{\sqrt{EA(s)-1}}\,\dd s
             \longrightarrow \frac{\pi w(s_*)}{\sqrt\kappa}
             \qquad
             (E\downarrow E_*).
          \end{equation}
          The local analytic coordinate $x=\sgn(s-s_*)\sqrt{1-\frac{A(s)}{A(s_*)}}$  satisfies $x'(s_*)=\sqrt\kappa$.
          Substituting $x=\sqrt{1-\frac{E_*}{E}}\sin\vartheta$ fixes the integration interval at $[-\frac{\pi}2, \frac{\pi} 2]$,
          and the transformed integrand converges uniformly to
          $\frac{w(s_*)} {\sqrt\kappa}$.
          This proves \eqref{eqath}.

          Second, let $p,r\geq 1$, $A_{p,r}(s)=\cos^{2p}s\sin^{2r}s$, and let $s_-(E)<s_+(E)$ solve $EA_{p,r}(s)=1$.
          For $\omega=\sec s$ or $\tan s$,
          \begin{equation}\label{eqahi}
             \int_{s_-}^{s_+}
                \frac{\omega(s)}{\sqrt{EA_{p,r}(s)-1}}\,\dd s
             \longrightarrow
             \int_1^\infty\frac{\dd y}{y\sqrt{y^{2p}-1}}
             =\frac{\pi}{2p}
             \qquad  (E\uparrow \infty).
          \end{equation}
          Split the integral at the maximizer $s_*$ of $A_{p,r}$.
          Near $s=0$, the coordinate $u=A_{p,r}(s)^{\frac 1 {2r}}$ has bounded inverse derivative.
          Rescaling $u$ by $E^{\frac{-1}{2r}}$ bounds the lower-half contribution by  $O(E^{-\frac 1 {2r}}\log E)=o(1)$; away from zero the integrand is $O(E^{-\frac 1 2})$.
          On the upper half,
               set $\lambda=E^{\frac{1}{2p}}$ and $y=\lambda\cos s$.
          Then $EA_{p,r}(s)=y^{2p}(1- \frac{y^2}{\lambda^2})^r$,
               while $\omega(s)\,\dd s$ becomes
                      $-\frac{\dd y}{y\sqrt{1-\lambda^{-2}y^2}}$ for $\omega=\sec s$
          and 
                    $-\frac {\dd y} y$ for $\omega=\tan s$.
          Through 
          factoring the simple zero at $y=\lambda\cos s_+\to1$ 
          one can get the convergence of the integrals on bounded intervals.
          For large $y$,
              both integrands are bounded by $Cy^{-p-1}$,  uniformly in $E$,
              so the tail beyond $R$ is bounded by $CR^{-p}$.
          This proves \eqref{eqahi}.

          \phantomsection
          \label{ap4}

       \begin{proof}[Proof of \cref{lm8}]
          Retain $p=k_1+1$.
          Suppose first that $k_2=0$.
          Analyticity follows from \cref{lm6},
          and
          $|\Delta_1(0,C_2)|
             =2\int_0^{s_+(C_2)}
             \frac{\sec s}{\sqrt{C_2\cos^{2p}s-1}}\,\dd s$.
          In the signed coordinate,
          this is the integral over $[-s_+,s_+]$.
          Applying \eqref{eqath} with
          $A(s)=\cos^{2p}s$,
          $s_*=0$,
          $w(s)=\sec s$,
          and $\kappa=p$
          gives the threshold limit $\frac{\pi}{\sqrt p}$.
          For the high-energy limit,
          set $\lambda=C_2^{\frac{1}{2p}}$ and $y=\lambda\cos s$.
          Then
          $$
             \frac{|\Delta_1(0,C_2)|}{2}
             =
              \int_1^\lambda
   \frac{             \dd y}   
       {y\sqrt{1-{\lambda^{-2}}{y^2}}\sqrt{y^{2p}-1}}
             \longrightarrow
             \frac\pi{2p}.
          $$
          Convergence on each fixed $[1,R]$ is immediate.
          Splitting the remaining interval at $\frac \lambda 2$
          bounds its contribution by $CR^{-p}$,
          uniformly for $\lambda\geq 2R$.

          Now suppose that $k_2>0$.
          Analyticity follows from \cref{pr5}.
          The maximum of $A_0(s)=\cos^{2p}s\sin^{2k_2}s$
          occurs at $\cos^2s_*=\frac p {p+k_2}$,
          with $\kappa=2(p+k_2)$.
          Applying \eqref{eqath} with $w=\sec s$
          gives $|\Delta_1|\to\pi\sqrt{\frac 2 p}$.
          Formula \eqref{eqahi},
          with $r=k_2$,
          gives $|\Delta_1|\to \frac \pi p$ at high energy.
          In either case,
          if $k_1+k_2>0$,
          the two limits are distinct,
          so the analytic function has nonzero derivative somewhere.
       \end{proof}

          \phantomsection
          \label{ap5}
       \begin{proof}[Proofs of \cref{lm15,lm16}]
          For $A(s)=\cos^{2p}s\sin^{2q}s$,
          the maximum occurs at $\tan^2s_*= \frac q p$,
          with $\kappa=2(p+q)$.
          The first formula in \eqref{eq47}, specialized to
          $C_1=-1$, and \eqref{eqath}, with
          $w=\tan s$, give
          \[
             \Delta_1\longrightarrow
             \frac{2\pi\tan s_*}{\sqrt{2(p+q)}}
             =\pi\sqrt{\frac{2q}{p(p+q)}}
             \qquad(C_2\downarrow P_*).
          \]
          The contact relation \eqref{eq74} gives the corresponding limit for
          $|\Delta_2|$.
          At high energy,
          \eqref{eqahi} applied to the same
          integral gives $\Delta_1\to\frac{\pi}{p}$,
          and \eqref{eq74} gives
          $|\Delta_2|\to\frac{\pi}{q}$.
          This proves both lemmas.
       \end{proof}

%\majorsectiongap
\section{Incidence calculations and examples}

          We give the incidence calculation and the two example verifications used in \cref{se20}.

   \subsection{Finiteness of incident phases}
          \label{ap9}

       \begin{proof}[Proof of \cref{lm18}]
           Identify $M$ with $L=f(M)$.
           By analytic regularity,
           $L$ is a compact real-analytic submanifold \cite{Morrey1966}.
           The compact analytic set
           \[
                              \Gamma_f
                              =\{(\zeta,a,b)\in\Sphere^1\times L\times L:
                              b=\zeta a\}
           \]
           admits a finite stratification by connected real-analytic manifolds  \cite[Proposition~2.10 and Corollary~2.11]{BierstoneMilman1988}.
           If $(i\upsilon\zeta,A,B)$ is tangent to a stratum at $(\zeta,a,b)$,
                 by differentiating $b=\zeta a$ 
                 we have  $B=\zeta(A+i\upsilon a)$.
           Since $L$ is $\cC$-totally real,
           \[
                              0=\langle B,ib\rangle_{\RR}
                              =\langle A,ia\rangle_{\RR}+\upsilon
                              =\upsilon.
           \]
           Thus the phase projection is constant on each stratum, and its image $\mathcal P_f$ is finite.
       \end{proof}

  \subsection{A shared factor for the spherical examples}
          \label{apnp}

       \begin{proof}[Verification of \cref{ex2,expp}]
           Take $p=1$,
           $q=3$ with $C_1=-1$.
           By \cref{lm15,lm16},
           continuity,
           and \eqref{eq74},
           choose a nonsteady contact profile with $(\frac{\Delta_1}{2\pi},    \frac{\Delta_2}{2\pi})=(\frac 9 {16},  -\frac 3 {16})$.
           Then its ordinary return order is $16$.
           Let $T_{\mathrm{in}}$ be its complete-cell time
           and $u=(u_1,u_2)=(e^{\frac {9\pi } 8 i},e^{- \frac {3\pi } 8 i})$ its cell-phase return.
           Up to a rotation, 
           %Coupling it to a point and the standard real sphere gives,
           %after a constant rotation,
            $Y(t,x)=(\gamma_1(t),\gamma_2(t)x)$, $x\in \Sphere^2$,
            defines a  special Legendrian embedding of $(\RR/(16T_{\mathrm{in}}\ZZ))\times\Sphere^2$ into $\Sphere^7$.
                  Here  \cref{lm12} excludes opposite-branch coincidences,
                    while any remaining coincidence is a $j$-cell shift with $u^j=(1,\varepsilon)$, $\varepsilon\in\{1,-1\}$,
                  and $\operatorname{ord}(u_1)=16$ forces $16\mid j$.
           Denote its image by $M$.

           Take $\zeta=e^{\frac {\pi } 4 i}$.
           Since $u^2=(\zeta,-\zeta)$,
           we have $\zeta Y(t,x)=Y(t+2T_{\mathrm{in}},-x)$,
           so $\langle\zeta\rangle\subset\mathsf H_M$
           and
           \begin{equation}\label{eqphi8}
                \varphi_\zeta([t,x])=[t+2T_{\mathrm{in}},-x].
           \end{equation}
           Every $\lambda\in\mathsf H_M$ satisfies $\lambda^4\in\{\pm 1\}$
           by the Lagrangian angle of the special Lagrangian cone $C(M)$.
           Hence $\mathsf H_M=\langle\zeta\rangle$.
           Formula \eqref{eqphi8} is a circle translation times the antipodal map of $\Sphere^2$,
           providing the homology actions used in \cref{expp}.

         Arrange $t=0$ at a lower turning point.
       On the outgoing half-cell,  $\theta_1-\theta_2$ is strictly increasing with total increment $\frac{3\pi} 4$.
           Thus a unique $t_*\in(0, \frac {T_{\mathrm{in}}} 2)$ where this increment is $\frac{\pi} 2$.
           Reflection gives  $Y(-t_*,-x)=\lambda_*Y(t_*,x)$  for some $\lambda_*\in S^1$.
        The corresponding profile phase difference is $(\lambda_*,-\lambda_*)$,
       so \cref{lm12} gives
       $\chi_{1,3}(\lambda_*,-\lambda_*) =-\lambda_*^4\notin\{\pm1\}$.
       Thus $\lambda_*^8\ne1$ and $\lambda_*\in\mathcal P_M\setminus\mathsf H_M$.

           For the steady product in \cref{ex2},
           take $z_-=Y(t_*,x)$ and $z_+=Y(-t_*,-x)$.
           If the tangent spaces at the two preimages were equal,
                    intersecting them with the first $\CC^4$ block and the Legendrian condition would give $T_{z_+}M=\lambda_*T_{z_-}M$.
           Again by the Lagrangian angle argument, one can get $\lambda_*^4\in\{\pm 1\}$  contradicting  $\lambda_*^8\neq1$.
           This proves the crossing asserted in \cref{ex2}.
       \end{proof}

%\majorsectiongap
\section{Projective Cover and Index Sectors}
          \label{ap11}

         We establish the finite covering used in \cref{co9} and the factor-sector identities used in \cref{th18}.
          With the notation of \cref{se21},
          use $g_i=\iota_i^*g_{\mathrm{FS}}$ for the induced metrics.

 \begin{lemma}[Projective-domain cover]
           \label{lm19}
           The profile projection makes $\mathcal X_\gamma$
           a smooth fiber bundle over $\mathcal C_\gamma\cong S^1$
           with fiber $(\widehat L_1\times\widehat L_2)/\mathsf D_{12}$.
           The natural map
           $\mathfrak P([\tau,x_1,x_2])
           =([\tau],\pi_1(x_1),\pi_2(x_2))$
           defines a $\kappa$-sheeted Riemannian covering
           $$
              \mathfrak P:\mathcal X_\gamma\longrightarrow
              \left(\mathcal C_\gamma\times L_1\times L_2,
              \dd\tau^2+\cos^2s(\tau)\,g_1
              +\sin^2s(\tau)\,g_2\right).
           $$
           Consequently,
           with $b_1(Y)=\dim H^1(Y;\RR)$,
           \begin{equation}\label{eq83}
              b_1(\mathcal X_\gamma)
              \geq  b_1(L_1)+b_1(L_2)+1.
           \end{equation}
       \end{lemma}

       \begin{proof}
           Over an $N$-cell interval,
           the domain is the product with the stated fiber.
        The endpoint gluing is induced by factor stabilizers,
       so in the block coordinates of \cref{rm:HC}
         it fixes $[Z_1]$ and $[Z_2]$
          and hence preserves $\pi_1$ and $\pi_2$.
           Thus the profile bundle may be twisted, but $\mathfrak P$ is well defined.
           On each fiber it is the product covering $\pi_1\times\pi_2$ modulo the diagonal group $\mathsf D_{12}$,
           so its degree is
           $|\mathsf H_{f_1}|\,|\mathsf H_{f_2}|/|\mathsf D_{12}|=\kappa$.
           The horizontal isometry of the Hopf projection
           and \eqref{eq7} show the stated metric.
          % Finally,
       %    pullback by $\mathfrak P$ is injective on real cohomology,
       %    and the K\"unneth formula gives \eqref{eq83}.
      %     
            If a closed $1$-form $\alpha$ on the base satisfies $\mathfrak P^*\alpha=du$,
       averaging $u$ over each fiber gives a function $\bar u$ on the base with $d\bar u=\alpha$.
       Thus pullback is injective on first real cohomology, and the K\"unneth formula proves \eqref{eq83}.
       \end{proof}

           For the factor sectors,
          let $\Delta_{0,i}$ be the positive Laplacian of $(L_i,g_i)$,
          and write
          $\mathfrak p_i=\operatorname{pr}_{L_i}\circ\mathfrak P$,
          $(\mathfrak a_1,\mathfrak a_2)=(\cos^2s,\sin^2s)$.
          Here
          $\operatorname{pr}_{L_i}$ denotes projection onto the $L_i$ factor.
          Relation \eqref{eq85} and  $\mathfrak a_1+\mathfrak a_2=1$ 
          give
          $\Delta_0\mathfrak a_i=2m\mathfrak a_i-2(n_i+1)$
          for $i=1,2$.
          For mean-zero functions $\phi,\psi$ on $L_i$,
          the lift
               $\mathcal T_i\phi=\mathfrak a_i\mathfrak p_i^*\phi$
               also has mean zero on $\mathcal X_\gamma$.
          The warped-product Laplacian gives
          \[
                              (\Delta_0-2m)\mathcal T_i\phi
                              =\mathfrak p_i^*
                              (\Delta_{0,i}-2(n_i+1))\phi.
          \]
          With  $\Delta_0(\mathfrak p_i^*h)  =\mathfrak a_i^{-1}\mathfrak p_i^*(\Delta_{0,i}h)$,
       we apply $\Delta_0$ to this identity
       and pair with $\mathcal T_i\psi$.
       The factors $\mathfrak a_i^{-1}$ and $\mathfrak a_i$ cancel,
       so Oh's formula and Fubini's theorem give
                 \[
                              Q_{\overline \iota_\gamma}^{\mathrm{Ham}}
                              (\mathcal T_i\phi,\mathcal T_i\psi)
                              =\mathfrak c_i
                              Q_{\iota_i}^{\mathrm{Ham}}(\phi,\psi),
                              \qquad \mathfrak c_i>0.
          \]
          Since $\int_{L_i}\Delta_{0,i}h\,\dd\vol_{L_i}=0$,
          the same calculation yields
          \[
                              Q_{\overline \iota_\gamma}^{\mathrm{Ham}}
                              (\mathcal T_1\phi,\mathcal T_2\psi)=0,
                              \qquad
                              Q_{\overline \iota_\gamma}^{\mathrm{Ham}}
                              (\mathcal T_i\phi,v(\tau))=0.
          \]
          Thus the two factor sectors and the profile sector are mutually orthogonal for $Q_{\overline\iota_\gamma}^{\mathrm{Ham}}$,
          and this induces the factor contributions in \cref{th18}.

%\majorsectiongap
          \vspace{\baselineskip}
          \raggedbottom

\end{document}